\documentclass[12pt]{article}

	\usepackage{amsmath, amssymb, amsthm, mathtools}
	\usepackage{hyperref, a4wide, color, graphicx, comment, dsfont}
	\usepackage{xstring, xparse}
	\usepackage[english]{babel}
	\usepackage{verbatim, booktabs, framed}
	\usepackage{autonum}

	\newtheorem{theorem}{Theorem}[section]
	\newtheorem{lemma}{Lemma}[section]
	\newtheorem{algo}{Algorithm}[section]

	\newtheorem{remark}{Remark}[section]
	
	\allowdisplaybreaks
	\newcommand{\transpose}{\top} %transpose symbol
	\newcommand{\tild}[1]{\tilde{#1}} %symbol for tilde
	\newcommand{\R}{\mathbb{R}} %real numbers
	\newcommand{\N}{\mathbb{N}} %natural numbers beginning with 1
	\newcommand{\linearspace}{\mathscr{L}} %symbol for linear/cont space between sets
	\newcommand{\diagmatrix}{\text{diag}} %symbol for diagonal matrix
	\newcommand{\diameter}{\text{diam}} %diameter of a set
	\newcommand{\closure}[1]{\overline{#1}} %closure of a set

        \definecolor{Blue}{rgb}{0,0,1}
        \def\corr#1{#1}% Hide it      
        \def\newcorr#1{#1}% Hide other chnagesit   
\title{{A parallel space-time $p$-adaptive discontinuous Galerkin method for nonlinear acoustics}}
\author{Daniele Corallo, Pascal Lehner, Christian Wieners}
\date{}

\begin{document}
\maketitle
\begin{abstract}
	In this paper, we introduce and analyze a space-time $p$-adaptive discontinuous Galerkin method for nonlinear acoustics. 

	We first present the underlying mathematical model, which is based on a recently derived formulation involving, in particular, only first order in time derivatives. We then propose a space-time discontinuous Galerkin discretization of this model, combining a symmetric Friedrichs systems discretization for symmetric hyperbolic systems with an interior penalty discretization for damping terms. The resulting nonlinear system is solved using Newton’s method.
	
	Next, we present a well-posedness analysis of the discrete problem. The analysis begins with a linearized system, for which stability is shown. Using a fixed point argument, these results are extended to the fully discrete nonlinear system, yielding a priori error estimates in a natural discontinuous Galerkin norm.
	
	Finally, we present numerical experiments demonstrating the parallel solvability of the space-time formulation and the effectiveness of $p$-adaptivity. The results confirm the theoretical convergence rates and show that adaptive refinement can reduce the number of degrees of freedom required to accurately approximate selected goal functionals. Moreover, the experiments demonstrate that the model reproduces characteristic phenomena of nonlinear acoustics, such as harmonic generation, thereby validating the proposed model.
\end{abstract}
%\tableofcontents
%
%
%
%
\newcommand{\G}{G}
\renewcommand{\linearspace}{\mathcal{L}}
\newcommand{\dt}{\partial_{t}}
\newcommand{\dx}[1]{\partial_{x_{#1}}}
\newcommand{\lapl}{\Delta}
\newcommand{\grad}{\nabla}
\newcommand{\divv}{\nabla \cdot}
\newcommand{\divvv}{\nabla \bullet}
\newcommand{\laplvec}{{\Delta}}
\newcommand{\uold}{\tild{u}}
\newcommand{\pold}{p^{\star}}
\newcommand{\vold}{v^{\star}}
\newcommand{\uvar}{u}
\newcommand{\pvar}{p}
\newcommand{\vvar}{\mathbf{v}}
\newcommand{\flux}{{F}}
\newcommand{\noflux}{{G}}
\newcommand{\fluxmat}{B}
\newcommand{\diffusion}{{A}}
\newcommand{\diffcoff}{\mu}
\newcommand{\reaction}{{C}}
\newcommand{\fvar}{f}
\newcommand{\hvar}{h}
\newcommand{\gvar}{g}
\newcommand{\embed}{\hookrightarrow}
\newcommand{\dime}{d}
\newcommand{\ndime}{{m}}
\newcommand{\indexx}{i}
\newcommand{\pdiff}{\mu}
\newcommand{\vdiff}{\eta}
\newcommand{\divcof}{\alpha}
\newcommand{\gradcof}{s}
\newcommand{\macha}{\varepsilon}
\newcommand{\machb}{ \varepsilon \beta}

\newcommand{\domain}{\Omega}
\newcommand{\spacetimedomain}{Q}
\newcommand{\testvar}{\phi}
\newcommand{\triang}{\mathcal{T}}
\newcommand{\trifaces}{\mathcal{F}}
\newcommand{\meshsize}{h}
\newcommand{\timesize}{n}
\newcommand{\starttime}{0}
\newcommand{\finaltime}{T}
\newcommand{\timesteps}{N}
\newcommand{\element}{K}
\newcommand{\face}{F}
\newcommand{\allelements}{\mathcal{K}_{\meshsize}}
\newcommand{\allfaces}{\mathcal{F}}
\newcommand{\allintfaces}{\trifaces^{\textrm{\corr{int}}}_{\meshsize}}
\newcommand{\allextfaces}{\trifaces^{\textrm{\corr{ext}}}_{\meshsize}}
\newcommand{\sobindex}{k}
\newcommand{\dummyvar}{w}
\newcommand{\dummyvarb}{\varphi}
\newcommand{\polynomialspace}{{{\mathcal{P}}}}
\newcommand{\brokenpoly}{S}
\newcommand{\dega}{{p}}
\newcommand{\degb}{{q}}
\newcommand{\inval}{I}
\newcommand{\Lelement}{\Lp{2}(\element)^{\ndime}}
\newcommand{\Lelementboundary}{\Lp{2}( \partial \element)^{\ndime}}
\newcommand{\Melement}{\Lp{2}(\element)^{\ndime \times \dime}}
\DeclarePairedDelimiter\timejump{ \{ }{ \} }
\DeclarePairedDelimiter\average{ \langle }{ \rangle }
\DeclarePairedDelimiter\jump{ [ }{ ] }
\DeclarePairedDelimiter\trace{ \text{Tr}( }{ ) }
\DeclarePairedDelimiter\disnorm{\lvert\!\lvert\!\lvert}{\rvert\!\rvert\!\rvert}

\newcommand{\p}{p}
\newcommand{\pv}{u}
\newcommand{\pvtest}{z}
\newcommand{\pvold}{\tild{\pv}}
\newcommand{\D}{D}
\newcommand{\NL}{N}
\newcommand{\h}{h}
\newcommand{\bilinfull}{b_{\h}}
\newcommand{\trilinspace}{c_{\h}}
\newcommand{\linfull}{\ell_{\h}}
\newcommand{\linbdr}{\ell_{\partial \Omega, \h}}
\newcommand{\nump}{\p_{\h}}
\newcommand{\numv}{v_{\h}}
\newcommand{\numptest}{q_{\h}}
\newcommand{\numvtest}{w_{\h}}
\newcommand{\numpv}{\pv_{\h}}
\newcommand{\numpvtest}{\pvtest_{\h}}
\newcommand{\bilintime}{m_{\h}}
\newcommand{\bilinspace}{c_{\h}}

\newcommand{\numpvold}{\tild{u}_h}
\newcommand{\numM}{M_{\h}}
\newcommand{\numspacetimedomain}{ {Q_{\h} }}
\newcommand{\numspacedomain}{ {\Omega_{\h}}}
\newcommand{\bilinspacesf}{a_{\h}}
\newcommand{\bilinspaceip}{d_{\h}}
\newcommand{\bilinspacenl}{n_{\h}}
\newcommand{\A}{A}
\newcommand{\Aup}{A_{\norvec_\element}^{\textrm{up}}}
\newcommand{\jacobian}{\boldsymbol{\nabla}}
\newcommand{\bilinspaceipbdr}{\rho_{\h}^\iptheta}
\newcommand{\gateau}{\delta}
\newcommand{\M}{M}

\newcommand{\norvec}{\mathbf{n}}
\newcommand{\radius}{\rho}
\newcommand{\ball}{\mathcal{B}}
\newcommand{\half}{\frac{1}{2}}
\newcommand{\Cshape}{C_R}
\newcommand{\Cbddeq}{C_1}
\newcommand{\Cbddeqq}{C_2}
\newcommand{\penality}{J}
\newcommand{\penpara}{\sigma}
\newcommand{\iptheta}{\Theta}
\newcommand{\bilina}{a}
\newcommand{\bilinb}{b}
\newcommand{\bilinc}{c}
\newcommand{\bilind}{d}
\newcommand{\bilinf}{f}
\newcommand{\biling}{g}
\newcommand{\bilinh}{h}
\newcommand{\bilinB}{B}
\newcommand{\linform}{l}
\newcommand{\numflux}{\mathcal{H}}
\newcommand{\error}{e}
\newcommand{\errora}{\xi}
\newcommand{\errorb}{\eta}
\newcommand{\Uvar}{U}
\newcommand{\proje}{\Pi}
%small greek letters for constants
%small latin letters for functions (vector or not)
%big latin letters for matrix (functions)
\newcommand{\Honezero}{H^1_0(\domain)}
\newcommand{\LtwoLtwo}{\Lp{2}([0,T];\Lp{2}(\domain)^\ndime )}
\newcommand{\fluxreg}{k}
\newcommand{\uregularity}{H^1_{\uvar_0}(\inval; \Lp{2}(\domain)^\ndime) \cap \Lp{2}(\inval;\Hs{2}(\domain)^\ndime \cap \Honezero^\ndime) \cap \Lp{\infty}(\inval;\Hs{1}(\domain)^\ndime) }
\newcommand{\vevar}{\phi}
\newcommand{\targetset}{V}
\newcommand{\vregularity}{\Lp{2}(I;\Hs{1}(\domain)^\ndime)}
\newcommand{\Uspace}{U}
\newcommand{\Vspace}{V}
\newcommand{\cylind}{Q}
\newcommand{\timesi}{n}
\newcommand{\discinval}{\mathcal{I}_n}
\newcommand{\spacetime}{\mathcal{R}}
\newcommand{\timeelement}{n}
\newcommand{\spacetimeelement}{R}
\newcommand{\dummyset}{S}
\newcommand{\localproje}{\Psi}
\newcommand{\rhs}{f}
\newcommand{\bilinges}{\vartheta}
\newcommand{\opges}{\varTheta}
\newcommand{\diffbound}{I}
\newcommand{\numfluxform}{H}
\newcommand{\pri}{'}
\newcommand{\opd}{D}
\newcommand{\opc}{C}
\newcommand{\opb}{B}
\newcommand{\opa}{A}
\newcommand{\testspace}{{Z}}
\newcommand{\numtestspace}{{Z}_{\meshsize}}
\newcommand{\seekspace}{{U}}
\newcommand{\numseekspace}{{U}_{\meshsize}}
\newcommand{\den}{\rho}
\newcommand{\sos}{c}
\newcommand{\ent}{s}
\newcommand{\para}{\alpha}
\newcommand{\parb}{\beta}
\newcommand{\parc}{\gamma}
\newcommand{\pard}{\delta}
\newcommand{\pare}{\varepsilon}
\newcommand{\parf}{\zeta}
\newcommand{\parg}{\eta}
\newcommand{\parh}{\theta}
\newcommand{\boundary}{\partial}
\newcommand{\timeint}{I}
\newcommand{\fintime}{T}
\newcommand{\settimes}{\times}
\newcommand{\f}{f}
\newcommand{\Non}{N}
\newcommand{\dir}{\Gamma_{\textrm{\corr{D}}}}
\newcommand{\neu}{\Gamma_{\textrm{\corr{N}}}}
\newcommand{\wrtdt}{\, \corr{\mathrm d}t}
\newcommand{\alldirfaces}{\allfaces^{\textrm{\corr{D}}}}
\newcommand{\allneufaces}{\allfaces^{\textrm{\corr{N}}}}
\newcommand{\ustar}{U_{\star, h}}
\newcommand{\faceset}{\mathcal{F}}
\newcommand{\Rem}{\mathcal{C}_h}
\newcommand{\AD}{\A_D}
\section{Introduction}
%Why is nonlinear acoustics important?
Efficient numerical methods for nonlinear acoustics are essential across various scientific and engineering disciplines. In engineering, nonlinear acoustic phenomena play 
a critical role in nondestructive 
material testing \cite{Fierro2015}, underwater acoustics \cite{CampoValera2023} and supersonic sources \cite{Sapozhnikov2019}. From a medical perspective, 
accurately modeling the propagation of nonlinear acoustic waves through 
biological tissue is essential for both diagnostic and therapeutic applications, 
including medical imaging, tissue ablation, lithotripsy, and targeted drug delivery, see \cite{rudenko_2022}. %We refer to \cite{Sapozhnikov2019} for an overview of the applications of nonlinear acoustics.

%What are the numerical challenges?
The simulation of nonlinear acoustics is often challenging due to the interplay between 
high frequency wave propagation and nonlinear effects. Many classical models, such as Westervelt’s~\cite{Westervelt1963} and Kuznetsov’s equation~\cite{kuznetsov:1971}, involve second order time derivatives that act on nonlinear terms. This significantly complicates the design of stable and efficient numerical methods. Moreover, a key difficulty lies in resolving the nonlinearities in the presence of steep gradients or shock like structures that often emerge during 
wave propagation, see \cite{Hamilton:1997} for more details on the physics of nonlinear acoustics. \\

\clearpage

%What has been done so far?
Numerous numerical methods have been proposed in the literature to address these challenges. Here, we focus on works that include an error analysis of a finite element method.  
For the Westervelt equation, semi-discrete finite element approximations with optimal error 
estimates have been studied in \cite{Nikolic2019}, while semi-discrete discontinuous 
Galerkin (DG) methods have been analyzed in \cite{Antonietti2020}. 
A time stepping DG approach using conforming spatial elements is presented in
\cite{Gomez2024}, and a hybridizable DG formulation is discussed in \cite{Meliani2024}. 
Efficient time integration techniques based on operator splitting are 
introduced in \cite{Kaltenbacher2015}.
For the Kuznetsov equation, a mixed finite element analysis is provided 
in \cite{Meliani2024}. Robust error estimates with respect to the damping parameter for a mixed method 
are derived in \cite{Dörich2024}, and coupling with elastic wave equations is 
addressed in \cite{Muhr2023}. 

A model for nonlinear acoustics using only first order time derivatives with a structure preserving discretization and conforming elements is developed in \cite{Egger2025}, 
with additional insights presented in \cite{Barham2024}.
Other numerical approaches, such as finite volume and finite difference methods, have also been explored, though often without convergence analysis. A finite volume method for Westervelt's equation has been investigated in, for example, \cite{VelascoSegura2015}, while a finite difference method has been studied in \cite{NortonPurrington2009}.
\\

%What are we doing and why?
In this paper, we propose a space-time DG method 
for a nonlinear acoustic model closely related to Kuznetsov's equation, 
derived in \cite{QuadraticWave}. The key distinction of this formulation to classical models is the use of only first order in time derivatives, and that the nonlinear terms only involve spatial derivatives. To the best of our knowledge, this work represents the first application of a space-time DG method to nonlinear acoustics.
The motivation for considering space-time methods lies in their intrinsic suitability for adaptive refinement, parallel 
implementation, and inverse problems. 
Advanced space-time DG methods have been devolved for wave equations in \newcorr{e.g. \cite{AntoniettiMazzieriMigliorini2020, corallo2022, Karabelas2016, langer:2019, bansal:2021}}, for parabolic problems in \cite{CangianiDongGeorgoulis2017, GomezMoiola2024}, and for Navier-Stokes equations in \cite{Klaij2006, RhebergenCockburnVanderVegt2013, TAVELLI2016}. We emphasize that this list of references is by no means exhaustive.
\\

%Discription of section contents
The remainder of the paper is organized as follows. In section \ref{se:model}, we introduce the continuous mathematical model and present the space–time DG discretization analyzed in this work. Section \ref{se:linearized} is devoted to the discrete well-posedness analysis of a linearized system. In section \ref{se:nonlinear}, these results are extended to the fully nonlinear discrete problem, yielding well-posedness and a priori error estimates for the proposed method. Numerical experiments validating the efficiency and possibility of a $p$-adaptive and parallel implementation are presented in section \ref{se:experiments}. Finally, concluding remarks and directions for future research are given in section \ref{se:conclusion}.

\clearpage

\section{First order in time model} \label{se:model}
In this section, we present the continuous model and its space-time DG discretization considered in this work. The derivation of this first order in time model is based on the standard assumption of small amplitude wave propagation and follows a methodology similar to that used in the derivation of Kuznetsov's equation. 
A key difference, however, is that one avoids taking additional derivatives of the reduced Navier–Stokes system, which are employed in the Kuznetsov or Westervelt formulation to further simplify the equations. For a detailed derivation, we refer the reader to \cite{QuadraticWave}.
\subsection{Continuous model}
Let $\domain \subset \R^\dime$ be an open, bounded space domain with dimension $d \in \{2,3\}$, and let $\Gamma := \boundary \domain$ denote the boundary of $\Omega$, which is assumed to be Lipschitz. Furthermore, let $\timeint := (0,\fintime) \subset \R$ denote a time interval with final time $T>0$, and define the corresponding space-time domain $\spacetimedomain:= \timeint \settimes \domain$.
The model is governed by parameters satisfying 
$$\para, \, \parb, \, \parc, \, \pard, \, \pare, \, \parf, \, \parg, \, \parh > 0, \, \parc \neq \pard,$$
and takes the form 
%see \cite[Section 2]{QuadraticWave} for a derivation,
\[ \label{eq:pvmodel}
	\begin{split}
	\para  \dt \p + \divv \vvar - \parb \lapl \p + \parc \p \divv \vvar + \pard \grad \p \cdot \vvar 
	&= \p_{S} \quad \, \text{in} \quad  \spacetimedomain 
	\\
	\pare \dt \vvar + \grad \p - \parf \lapl \vvar + \half \grad \! \left(\parg \vvar^2 - \parh \p^2 \right) 
	&= \vvar_{S}	\quad \, \text{in} \quad \spacetimedomain 
	\\
	\pvar = \p_0, \quad \vvar 
	&= \vvar_0 \quad \,\text{on}\quad \{ 0 \} \settimes \domain 
	\\
	\pvar = \p_{D}, \quad \vvar 
	&= \vvar_{D} \quad\text{on}\quad \timeint \settimes \dir 
	\\
	\parb \grad \pvar \cdot \norvec = p_{N}, \quad \vvar 
	&= \vvar_{N} \quad\text{on}\quad \timeint \settimes \neu
	\end{split}
\]
with $\dir \cup \neu = \Gamma$, where $\dir \cap \neu = \emptyset$, and $\dir$ has positive $(d-1)$-dimensional Lebesgue measure. 
In this model of nonlinear acoustics $\p:\spacetimedomain \to \R$ can be understood as the fluctuation of pressure and $\vvar : \spacetimedomain \to \R^\dime$ as the fluctuation of the velocity field, respectively, cf. \cite[Section 2]{QuadraticWave} for details. By $\norvec$ we denote the outer unit normal vector on $\boundary \domain$ and $\cdot^\transpose$ means transposition of a vector or matrix. We write $\vvar^2:= \vvar \cdot \vvar$ for the squared norm of a vector, $\grad := (\dx{1}, \dots, \dx{d})^\transpose$ for the vector of spatial partial derivatives, $\lapl := \divv \grad$ for the Laplace operator, which acts component wise on vector fields, and $\jacobian$ for the Jacobian matrix of a vector field.

In order to simplify notation, we rewrite equation \eqref{eq:pvmodel} with $\uvar = (u_p, u_\vvar )^\transpose = (\p, \vvar)^\transpose \in \R^{1+\dime}$ as
\[\label{eq:umodel}
\begin{split}
\left(\M \dt + \A - \D \lapl \right) \! \uvar + \Non( \uvar, \uvar) &= \f \, \, \, \, \quad \text{ in } \quad Q \\
\uvar &= u_0 \quad \,  \text{ on } \quad \{ 0 \} \settimes \domain \\
\mathcal{B}_D u &= u_D \quad \text{ on } \quad \timeint \settimes \dir \\
\mathcal{B}_N u &= u_N \quad \text{ on } \quad \timeint \settimes \neu
\end{split}
\]
with matrices $
\M := \diagmatrix(\para, \pare, \dots, \pare), \ \D:=\diagmatrix(\parb, \parf, \dots, \parf) \in \R^{(1 + \dime) \times ( 1 + \dime)} $,
first order differential operators  
\[
\begin{split}
 \A u := \begin{pmatrix}
\divv u_\vvar \\ \grad u_p
 \end{pmatrix}, \quad 
 \Non(\pv, z) := 
 \begin{pmatrix}
\parc z_p \divv \pv_\vvar + \pard \grad \pv_p \cdot z_\vvar \\
\parg \left(\jacobian \pv_\vvar \right)^\transpose \! z_\vvar - \parh z_p \grad \pv_p
 \end{pmatrix},
\end{split}
\]
boundary operators $\mathcal{B}_D u = u, \ 
\mathcal{B}_N u = ( \grad u_p \cdot \norvec, u_\vvar)^\transpose$,
and given data \[ \label{eq:data}
\begin{split} 
 \f&:=(\pvar_{S}, \vvar_{S})^\transpose \in L^2(\spacetimedomain;\R^{1 + \dime}), \\
 u_0 &:= (p_0, \vvar_0)^\transpose \in H^1(\domain;\R^{1+d}), \\
 u_D &:= (p_{D}, \vvar_{D})^\transpose \in H^1(0,T;H^{\frac{3}{2}}(\dir;\R^{1+d}) ), \\
 u_N&:=(p_{N}, \vvar_{N})^\transpose \in H^1(0,T; H^{\half}( \neu) \settimes H^{\frac{3}{2}}(\neu;\R^{d}) ).
\end{split} 
 \]
In \cite{QuadraticWave, QuadraticWave2} it is shown that equation \eqref{eq:umodel} is well posed for any $T>0$ with  
\[
\uvar \in U:=H^1(0,T; L^2(\Omega;\R^{1+\dime})) \cap L^2(0,T;H^2(\Omega;\R^{1+ \dime})) 
\]
assuming that $\dir = \boundary \domain$ or $\neu = \boundary \domain$ is $C^2$-regular, and that the data \eqref{eq:data} are sufficiently small with respect to the corresponding norms.
\corr{Note that in our analysis the case $\neu = \boundary \domain$ is not included.}

This is established as follows. Writing equation \eqref{eq:umodel} as
\[\label{eq:G}
\G(u) := L u - \D \lapl u + \Non(u, u) - f = 0
\]
with $L:= \M \dt + \A$, it is shown in \cite{QuadraticWave, QuadraticWave2} that there exists an $r>0$ such that Newton's method with starting value $0$ converges to a unique $u \in B_r^U(0)$ such that $G(u)=0$. Here, for any normed space $X$, $r>0$ and $x \in X$ we denote the open ball with radius $r$ centered at $x$ as $B^X_r(x)$. 
In the analysis, it becomes necessary to compute the Fréchet derivative of $\G$ in $U$ at $\pvold$. Owing to the bilinearity of $\Non(\cdot,\cdot)$ and the linearity of all other operators in $G$, one can demonstrate through suitable estimates on $\Non(\cdot,\cdot)$, that $\G$ is Fréchet differentiable in $U$. 
Then, Newton's method requires the linear system 
\[ \label{eq:linearcont}
\begin{split}
	(L - \D \lapl ) \pv + \NL(\pv, \pvold) + \NL(\pvold, \pv) &= f^{\pvold} \quad  \, \text{ in } \quad Q\\
	\uvar &= u_0 \quad \, \text{ on } \quad \{ 0 \} \settimes \domain\\
	\mathcal{B}_D u &= u_D \quad  \text{ on } \quad \timeint \settimes \dir \\
	\mathcal{B}_N u &= u_N \quad \text{ on } \quad \timeint \settimes \neu,
\end{split}
\]
to be solved in each iteration, where in the context of Newton's method $\f^{\pvold}$ plays the role of a residual and $\pvold \in U$ denotes the previous iterate. To establish global in time well-posedness of equation \eqref{eq:linearcont}, the data \eqref{eq:data} and $f^{\pvold}$ are assumed to be small enough. Once this is proven, the Newton-Kantorovich Theorem, see Lemma \ref{NewtonKantorowich} below, can be applied to $\G$. This gives global in time well-posedness of equation \eqref{eq:umodel}, if the data \eqref{eq:data} are sufficiently small. For showing the well-posedness of the discrete system, we follow a similar strategy, see section \ref{se:nonlinear}.
\subsection{Space-time discretization}
%In this section, we define the space-time DG method for equation \eqref{eq:umodel} considered in this paper. 
To begin, we introduce some notation. Let $S \subset \R^s$ be Lipschitz with $s \in \N$. We use the abbreviation $L^2=L^2(S)=L^2(S;\R^i)$ with $i \in \N$ when the context is clear. The standard $L^2$ inner product of two functions $f,g \in L^2(S;\R^i)$ is given by $(f, g)_S := \int_S f \cdot g\, d\sigma$, where $\sigma$ is a measure on $S$. If $f,g$ are matrix-valued, $ f\cdot g$ corresponds to the Frobenius inner product. 
%In the notion $\grad$ we do not distinguish between broken or continuous derivative operators.
%\subsubsection{Space-time discretization}
%\subsubsection*{Space-time discretization}
Let $k \in \N$ denote the number of time points of the considered discretization. For given time points $0 =: t_0 < t_1 < \ldots < t_{k-1}< t_k := T$, we define the following time intervals and time step sizes for $j \in \{ 1, \ldots, k\}$.
\[ 
\timeint_{j} := (t_{j-1}, t_j), 
\quad \timeint_h := \bigcup_{j =1}^k I_j \subset \timeint, 
\quad \boundary I_h := \bigcup_{j = 0}^k\{ t_j \}, 
\quad \tau_j := t_j - t_{j-1}, 
\quad \tau := \max_{j \in \{ 1, \dots, k\}}
%_{n \in \{1, 2, \dots, N\} } 
\tau_j . 
\] 
In this paper, we assume a quasi-uniform time discretization, i.e., $\tau_j \in [C_{sr} \tau, \tau]$ for all $j \in \{1, \dots, k \}$ with a shape regularity constant $C_{sr} \in (0,1]$ independent of $k$ and $j$.

For the spatial discretization, let $\allelements$ be a conforming shape-regular mesh of $\domain$ so that
${\domain}_h := \cup_{\element \in \allelements} {\element}$ is a decomposition of $\domain$ into open non overlapping space cells $\element \subset \corr{\Omega} %or \element \in  \allelements 
$. We denote spacial mesh sizes as $\meshsize_\element:= \diameter(\element)$ and $\meshsize := \max_{\element \in \allelements} \meshsize_\element$.

For simplicity, we consider only tensor-product space-time meshes in this article. To obtain a tensor-product mesh of $\spacetimedomain$, we define
\[
\spacetimedomain_h := \timeint_h \settimes \domain_h = \bigcup_{\spacetimeelement \in \spacetime_h} R
\text{ with } \spacetime_h:= \{ R \subset Q : \exists j \in \{1,\ldots,k\}, \exists \element \in \allelements :\spacetimeelement = \timeint_j \settimes K \}, 
\]
such that $\overline{Q_h} = \overline{Q}$.
We assume that the mesh sizes in time and space are well balanced satisfying for some $\omega \in [1,\infty)$
\[ \label{eq:mesh_balance} h^\omega \leq c_{r} \tau \leq h < 1, \] where $c_{r} > 0$ is a reference velocity yielding the bound $\tau^{-1} \leq c_r h^{-\omega}$.
\subsubsection*{Element boundaries}
We denote by $\allfaces_K$ the set of faces of a space element $\element \in \allelements$, and define the set of all faces as $ \allfaces_h := \bigcup_{\element \in \allelements} \allfaces_K$, where $\meshsize_\face = \diameter(\face)$ for $\face \in \allfaces_h$. The space skeleton is then given by $\boundary \domain_h = \closure{\corr{\bigcup}_{\face \in \allfaces_h}  \face }$, and the corresponding space-time skeleton is $\boundary Q_h = \boundary \timeint_h \settimes \closure{\Omega} \cup \closure{I} \settimes \boundary \domain_h$.

We distinguish between boundary faces $\allextfaces:= \{ \face \in \allfaces_h : \face \subset \partial \domain \} $ and interior faces $ \allintfaces := \allfaces_h \setminus \allextfaces$. Additionally, we define Dirichlet faces as $\alldirfaces_h:=  \{ \face \in \allfaces_h : \face \subset \dir \}$ and Neumann faces as $\allneufaces_h:= \{ \face \in \allfaces_h : \face \subset \neu \}$.  We assume that the mesh $\allelements$ is such that $\allextfaces = \alldirfaces_h \cup \allneufaces_h $ disjointly with $\alldirfaces_h \subset \dir$, $\allneufaces_h \subset \neu$ such that $\closure{\bigcup_{\face \in \alldirfaces_h} \face} = \dir$ and $\closure{\bigcup_{\face \in \allneufaces_h} \face} = \neu$. 
%Since we assume that the Lebesgue measure of $\dir$ is non zero, 
Furthermore, we require that the $(d-1)$-dimensional Lebesgue measure of $\alldirfaces_h$ is positive.

If an interior face $F \in \allintfaces$ is contained in $\element \in \allelements$, we write $\element_\face$ for the unique neighbor cell such that $F= \boundary \element \cap \boundary \element_\face$. For an exterior face $F \in \allextfaces \subset \element \in \allelements$, we set $\element_\face = \element$.
Each element $\element \in \allelements$ is assigned the outer normal vector $ \norvec_K$ of $\boundary \element$.
\subsubsection*{Approximation spaces}
Let $i,s \in \N$. We define the space of polynomials $\polynomialspace(\dummyset;\R^i) := \bigcup_{\deg=0}^\infty \polynomialspace^{\deg}(\dummyset;\R^i)$, where $\polynomialspace^{\deg}(\dummyset;\R^{i})$ is the set of polynomials $\dummyvarb: \dummyset \to \R^{i}$ with degree less or equal then $\deg \in \N_0$ for any set $\dummyset \subset \R^s$.
The approximation spaces considered here are broken polynomial spaces, namely, subsets of
\[ 
\polynomialspace_h(\dummyset;\R^{i}) := \left\{ \dummyvarb \in L^2(S;\R^i) : \dummyvarb \vert_{S} \in \polynomialspace(S_h;\R^i) \ \forall S_h \in \mathcal{S}_h  \right\} ,
\]
where $\mathcal{S}_h$ is some mesh of $S$.

To construct space-time tensor product broken polynomial spaces, we assign polynomial degrees $\dega_\spacetimeelement \in \N_0$ in time and $\degb_\spacetimeelement \in \N_0$ in space to every space-time cell $\spacetimeelement \in \spacetime_h$. 
This results in the discontinuous finite element space
\[ \label{eq:def_Zh}
\testspace_{h}:=  \left\{ 
\dummyvarb \in L^2(\spacetimedomain;\R^{1+\dime}) : \dummyvarb \vert_{ R=I_j \times \element} \in \polynomialspace^{\dega_R}(I_j;\R^{1+\dime}) \otimes \polynomialspace^{\degb_R}(\element;\R^{1+\dime}) \ \forall \spacetimeelement \in \spacetime_h 
\right\},
\]
which is used in this paper. Here, $\otimes$ denotes the tensor product between vector spaces.
\subsubsection*{Boundary operators}
Let $\dummyvar \in \left( H^1(I;L^2(\Omega;\R^i)) \cap L^2(I;H^1(\Omega;\R^i)) \right) \cup \polynomialspace_h(Q,\R^i)$ with $i \in \N \cup (\N \times \N)$ be such that traces in time and space of $w$ exist on $\spacetime_h$. For $\element \in \allelements$ and $j \in \{ 1, \ldots, k\}$, we define $\dummyvar_{j} := \dummyvar \vert_{I_j}$ and $\dummyvar_{K}:= \dummyvar \vert_{K}$, where the evaluation of these expressions at boundary points is understood in the sense of traces. In particular, this gives
\[
w_j(t_j) = \lim_{t \to t_j^+ } w_j(t), \quad w_j(t_{j-1}) = \lim_{t \to t_{j-1}^- } w_j(t) ,
\]
if the limits exist.

Let $\element \in \allelements$ be a space cell, $\face \in \faceset_\element \cap  \allintfaces$, $\element_\face$ the neighboring cell. We define the (space) jump at $\face$ with respect to $\element$ as
\[
\jump{\dummyvar}_{\face, \element}(t,x) := \dummyvar_{ \element_\face}(t,x) - \dummyvar_{ \element}(t,x) \quad x \in \face, \ t \in \timeint.
\]
The jump with respect to time is defined for each $j = 1, \dots, k-1$ as
\[
\jump{\dummyvar}_j(x) := \dummyvar_{{j+1}}(t_j,x) - \dummyvar_{{j}}(t_j,x) \quad x \in \domain,
\] 
and for the boundary cases $j=0,k$, we set $\jump{w}_0(\cdot) := w(t_0, \cdot)$, $\jump{w}_k(\cdot) := - w(t_k, \cdot) $.

Additionally, we use the notion of a (space) jump at faces $\face \in \faceset_h$ without a particular reference cell. For this, we define $\norvec_{\face}$ as the unit outer normal vector of $\face$. If $\face \in \allextfaces$, $\norvec_{\face}$ is taken to have the same orientation as $\norvec_{K}$ with $\face \subset K$. Otherwise, the orientation of $\norvec_\face$ may be chosen arbitrarily. Let $\face$ and $\face_1$ be adjacent faces with $\face \subset K, \face_1 \subset K_1$.
The (space) jump at $\face$ is defined as
\[
\jump{\dummyvar}_F(t,x) := \dummyvar_{K}(x,t) - \dummyvar_{K_1}(x,t) \quad x \in \face \cap \face_1, \ t \in \timeint
\]
and the (space) average at $\face$ is
\[
\average{\dummyvar}_F(t,x) := \half \left( \dummyvar_{K}(x,t) + \dummyvar_{K_1}(x,t) \right) \quad x \in \face \cap \face_1, \ t \in \timeint.
\]
On boundary faces $\face \in \allextfaces$, we set $\jump{\dummyvar}_F = \average{\dummyvar}_\face = \dummyvar \vert_F$. 
\subsection{Discontinuous Galerkin scheme}
We now introduce the DG scheme for equations \eqref{eq:umodel} and \eqref{eq:linearcont} considered in this article. DG space-time discretizations of $ \M \dt + \A $ can be framed within the theory of symmetric Friedrichs systems, see, e.g., \cite[Chapter 7]{dipietro2012}. In this context, we build upon the primal formulation of a Friedrichs system and then add terms corresponding to an interior penalty method for the Laplace operator $-\lapl$, see, e.g., \cite{dipietro2012, dolejsi_discontinuous_2015}, along with nonlinear volume terms. %The resulting scheme is similar to DG schemes for advection-convection equations, as described in \cite{dipietro2012}.

For convenience, we do not change the notation of differential operators, when applied to  globally discontinuous functions. In these cases, we mean the broken operator, which is defined piecewise on each (space-time) cell.
\subsubsection*{Symmetric Friedrichs system}
To discretize $\dt \M + \A$, as in \cite{DoerflerFindeisenWienersZiegler:2019}, we consider the following bilinear forms 
defined on $Z_h \times Z_h$, cf. \eqref{eq:def_Zh},
\[ 
\bilintime( \numpv, \numpvtest) :=  \sum_{j=1}^k \left( \int^{t_j}_{t_{j-1}} (\M \dt \numpv, \numpvtest )_{\numspacedomain} \wrtdt + \left(\M \jump{ \numpv}_{j-1}, z_{h,j}(t_{j-1}) \right)_{\numspacedomain} \right)
\]
and
\[  \label{eq:def_ah}
\bilinspacesf( \numpv , \numpvtest) := \int_0^T \sum_{\element \in \allelements} \! \left(  ( \A \numpv, \numpvtest)_{\element} + \sum_{\face \in \faceset_\element} (\Aup \jump{ \numpv }_{\face, \element}, z_{h,K})_F  \right) \wrtdt,
\]
where %with $Z_0:=\sqrt{\frac{\pare}{\para}}$ 
\[ \label{eq:def_upwindflux}
\Aup := \half
\begin{pmatrix}
	- \sqrt{\frac{{\alpha}}{{\varepsilon}}} & \norvec_K^\transpose \\
	\norvec_K & - \sqrt{\frac{{\varepsilon}}{{\alpha}}} \norvec_K \norvec_K^\transpose 
\end{pmatrix}
 \in \R^{(1+\dime) \settimes (1+\dime)} 
 \] 
is related to an upwind flux based on solving local Riemann problems, see section \ref{se:upwindflux} for the derivation of $\Aup$ for $d=2$. 

In \eqref{eq:def_ah}, the jump terms $\jump{u_h}_{F,K}$ are not defined for boundary faces $F \subset \Gamma$. To treat the boundary conditions correctly, we adopt the following convention
\[ \label{eq:def_jumpboundary}
\begin{cases}
	\jump{u_h}_{F,K} :=  -   { \numpv } \text{ for }  \face \in \dir , \\
	\jump{u_h}_{F,K} :=  - 2 (0,  { \numv } )^\transpose \text{ for }  \face \in \neu 
\end{cases} 
\]
with $u_h = (p_h , v_h)^\transpose$.

\corr{
As it is shown in \cite{corallo2022} (and also  is proven in section \ref{se:upwindflux}),
the definition \eqref{eq:def_jumpboundary} ensures the bounds
\[ \label{eq:upwindflux_assump}
c_\A  \int_0^T \| A_\norvec \jump{z_h} \|_{L^2(\partial \Omega_h)}^2 \wrtdt \leq 
a_h(z_h,z_h) 
\leq C_\A  \int_0^T \| A_\norvec \jump{z_h} \|_{L^2(\partial \Omega_h)}^2 \wrtdt \quad \forall z_h \in Z_h
\]
with \[ 
\| A_\norvec \jump{z_h}(t) \|_{L^2(\partial \Omega_h)} := \left(\sum_{\element \in \allelements} \sum_{\face \in \faceset_K} \|\ A_\norvec \jump{z_h}(t)\|_\face \right)^{\half}, \quad
%\| \cdot \|
A_{\norvec_K} := \begin{pmatrix}
	0 & \norvec_K^\transpose \\
	\norvec_K & 0 
\end{pmatrix},
\] where $c_\A, C_{\A}>0$ only depend on $\para$ and $\pare$ and $\Aup$
is defined in \eqref{eq:def_upwindflux}.
}

\subsubsection*{Interior penalty method and nonlinear terms}
To discretize the Laplacian terms $- D \lapl$, we use the interior penalty method, see, for example, \cite{dipietro2012, dolejsi_discontinuous_2015} and the references therein.
The associated bilinear form is given by
\[\label{eq:interior_penalty}
\bilinspaceip( \numpv , \numpvtest) := ( \D \jacobian \numpv, \jacobian \numpvtest)_{Q_h} + \int_0^T  \bilinspaceipbdr(\numpv, \numpvtest) \wrtdt
\]
with
\[
\begin{split}
 \bilinspaceipbdr(\numpv, \numpvtest) := &\beta \!\!\! \sum_{\face \in \allfaces_h \setminus \allneufaces_h} \!\!\! - \iptheta ( \jump{\nump}_F, \average{ \grad \numptest}_F  \cdot \norvec_\face )_\face - ( \average{ \grad \nump}_F \cdot \norvec_\face, \jump{ \numptest}_F )_\face + \penpara_{\face, h} (  \jump{\nump}_F, \jump{\numptest}_F)_\face   \\
  &+ \zeta \! \sum_{\face \in \allfaces_h} \!  - \iptheta ( \jump{\numv}_F, \average{ \jacobian \numvtest}_F  \norvec_\face )_\face - ( \average{ \jacobian v}_F \norvec_\face, \jump{ \numvtest}_F )_\face  + \penpara_{\face, h} (  \jump{\numv}_F, \jump{\numvtest}_F)_\face ,
\end{split}
\]
\corr{using $z_h = (q_h, w_h)^\transpose$ and depending on 
a penalty parameter $\penpara_{\face, h}>0$ and
$\Theta \in \{ 1,-1,0 \}$ corresponding to the symmetric, asymmetric and incomplete interior penalty methods, respectively.}
Note that in $\rho_h^\iptheta$ we only sum over boundary faces, when Dirichlet boundary conditions are applied.

To ensure a stable scheme, the penalty parameter $\penpara_{\face, h}:= \frac{C_\sigma}{h_F}$ must be chosen sufficiently large, see, e.g., \cite[Corollary 2.41]{dolejsi_discontinuous_2015} for lower bounds on $C_\sigma>0$. Additionally, $C_\sigma$ should be selected based on the local spacial polynomial degree near $F$. For further details, we refer to \cite[Section 7]{dolejsi_discontinuous_2015} and the references therein.

\corr{For} the nonlinear terms, we consider only volume elements, so that 
\[
\bilinspacenl( \numpv, \numpvtest) := ( \NL( \numpv, \numpv ), z_h )_{\numspacetimedomain} = \half n'_h(u_h,z_h;u_h)
\]
is chosen to discretize the nonlinear terms. For the linearized equation \eqref{eq:linearcont}, we obtain
\[
n'_h( \numpv, \numpvtest; \numpvold) := ( \NL( \numpv, \numpvold) + \NL(\numpvold, \numpv), \numpvtest  )_{\numspacetimedomain}
\]
with $\numpvold \in Z_h$. We remark that $n'_h(\cdot, \cdot;\tild{u}_h)$ is the Fréchet derivative of $n_h(\cdot, \cdot)$ at $\numpvold$.
\subsubsection*{The complete schemes}
The complete nonlinear discrete variational problem approximating equation \eqref{eq:umodel} is then given by
\[ \label{eq:full_scheme_nonlinear}
\text{solve for } \numpv \in \numtestspace \text{ such that } b_h(\numpv , \numpvtest) = \linfull(\numpvtest) \text{ for all } z_h \in Z_h
\]
with 
\[
b_h(\numpv , \numpvtest) := m_h(\numpv , \numpvtest) + a_h(\numpv , \numpvtest) + \bilinspaceip(\numpv , \numpvtest) + \bilinspacenl(\numpv , \numpvtest)
\]
and
\[ \label{eq:l_h}
\linfull( \numpvtest) := (\rhs_h, \numpvtest)_{\numspacetimedomain} + (\M \pv_{0,h}, \numpvtest(0))_{\numspacedomain} + \int_0^\fintime \linbdr( \numpvtest) \wrtdt .
\]
Depending on the boundary conditions, we have 
\[ \label{eq:l_h_boundary}
\begin{split}
	\linbdr(\pvtest_h) := \sum_{\face \in \alldirfaces_h}  &-(A^{\textrm{up}}_\norvec u_{D,h}, \numpvtest)_F  - \iptheta ( \D u_{D,h}, \jacobian \numpvtest  \norvec_\face )_\face + \penpara_{\face, h}( \D {u}_{D,h}, \numpvtest)_F
	\\
	&+ \sum_{\face \in \allneufaces_h} 
	-(2 A^{\textrm{up}}_\norvec (0,v_{N,h})^\transpose, \numpvtest)_F  - \iptheta ( \D (0,v_{N,h})^\transpose, \jacobian \numpvtest  \norvec_\face )_\face \\ &+ \penpara_{\face, h} ( \D (0,v_{N,h})^\transpose, \numpvtest)_F +
	((p_{N,h},0)^\transpose, \numpvtest)_F,
\end{split}
\]
where $f_h, u_{0,h}, u_{D,h}, u_{N,h} \in Z_h$ denote suitably interpolated data. Note that the added terms in \eqref{eq:l_h_boundary} ensure consistency, see Lemma \ref{le:consistent}. The complete linear variational problem discretizing equation \eqref{eq:linearcont} is given by
\[ \label{eq:full_scheme}
\text{solve for } \numpv \in \numtestspace \text{ such that } b_h'(\numpv , \numpvtest; \numpvold) = \linfull(\numpvtest) \text{ for all } z_h \in Z_h
\]
with 
%\[
$b_h'(\numpv , \numpvtest; \numpvold) = m_h(\numpv , \numpvtest) + a_h(\numpv, \numpvtest) + \bilinspaceip(\numpv , \numpvtest) + n_h'(\numpv , \numpvtest; \numpvold).
$ %\]
 
To solve the nonlinear problem \eqref{eq:full_scheme_nonlinear}, we apply a Newton-Raphson method, see \cite{Deuflhard2011}, described in Algorithm \ref{algo:newton}, solving the linear problem \eqref{eq:full_scheme} in each iteration. To this end, we define the continuous residual as
$%\[
r(u) := \G(u) 
$ %\]
\corr{using \eqref{eq:G}}, and the discrete residual as
\[
r_h({u_h}) := b_h(u_h, \cdot)  - \ell_h .
\]
Fixing a tolerance constant $c_\textrm{tol}>0$ and an initial guess $u_h^{(0)}$, the following algorithm generates a sequence $u^{(n)}_h \in Z_h$, $n \in \N$.
\begin{algo}[Newton's method] \label{algo:newton}
\begin{framed}
Set $r_h^{(0)} = r_h({u_h^{(0)}})$ and set $n=0$.\\
While $\| r_h^{(n)} \|_{L^2(Q)} > c_\textrm{\rm tol}$ do:
\begin{enumerate}
\item Solve for $u_h \in Z_h$ such that $b_h(u_h, z_h; {u_h^{(n)}}) = r_h^{(n)}(z_h) \ \forall z_h \in Z_h$. 
\item Update $u^{(n+1)}_h = u^{(n)}_h - u_h$, $r_h^{(n+1)} = r_h({u^{(n+1)}_h})$, $n \to n +1$.
\end{enumerate}
\end{framed}
\end{algo}
\clearpage
\section{Linearized system} \label{se:linearized}
In this section, we treat the well-posedness and convergence analysis of the linear discrete problem \eqref{eq:full_scheme}  under the condition that $\tild{u}_h$ is sufficiently small with respect to a suitable norm.
%We start with some preliminaries.
\subsection{Preliminaries}
%For the following analysis, as in \cite{corallo2022}, we use
%\[ \label{eq:upwindflux_assump}
%c_\A  \int_0^T \| A_\norvec \jump{z_h} \|_{L^2(\partial \Omega_h)}^2 \wrt\,\mathrm dt \leq 
%a_h(z_h,z_h) 
%\leq C_\A  \int_0^T \| A_\norvec \jump{z_h} \|_{L^2(\partial \Omega_h)}^2 \wrt\,\mathrm dt \quad \forall z_h \in Z_h\]
%with \[ 
%\| u_h(t) \|_{L^2(\partial \Omega_h)} := \left(\sum_{\element \in \allelements} \sum_{\face \in \faceset_K} (u_h(t), u_h(t))_\face \right)^{\half}, \quad
%\| \cdot \|
%A_{\norvec_K} := \begin{pmatrix}
%	0 & \norvec_K^\transpose \\
%	\norvec_K & 0 
%\end{pmatrix},
%\] where $c_\A, C_{\A}>0$ only depend on $\para$ and $\pare$ and $\Aup$ is defined in \eqref{eq:def_upwindflux}. Relation \eqref{eq:upwindflux_assump} is proven in section \ref{se:upwindflux}.
%Moreover, w

We assume that $\sigma_{F,h}$ is chosen sufficiently large such that
\[  \label{eq:coersive_d}
d_h(z_h,z_h) \geq c_{\iptheta} \int_0^T \disnorm{z_h(t)}_h^2 \wrtdt \quad \forall z_h \in Z_h,
\]
where $c_{\iptheta}>0$ is independent of $h$ and $\disnorm{\cdot}_h$ is defined below in \eqref{eq:def_ipnorm}. We refer to, e.g., \cite[Corollary 2.41]{dolejsi_discontinuous_2015} for a proof of \eqref{eq:coersive_d} and possible choices of \corr{the penalty parameter
  $\sigma_{F,h}$ in  \eqref{eq:interior_penalty}}.

The analysis is carried out with respect to the $h$-dependent norm 
	\[ \label{eq:DGNorm}
        \| u_h \|_{Z_h} := \left(\half \sum_{j=0}^k \| \M^{\frac{1}{2}} \jump{u_h}_{j} \|_{L^2(\Omega_h)}^2 + \int_0^T c_{\A} \| A_\norvec \jump{u_h} \|_{L^2(\partial \Omega_h)}^2 + c_{\iptheta} \disnorm{ u_h}^2_h
        \wrtdt \right)^\half \text{ in } Q\,,
\]
where with $u_h(t) = u_h=(p_h,v_h)^\transpose$ \corr{for fixed $t$}
\[ \label{eq:def_ipnorm}
\disnorm{u_h}_h:= \left( \| D^\half \jacobian u_h \|_{L^2(\Omega_h)}^2 + \beta \! \sum_{\face \in \allfaces_h \setminus \allneufaces_h} \penpara_{F,h} \| \jump{p_h}_F \|^2_{L^2(\face)} + \zeta \! \sum_{\face \in \allfaces_h} \penpara_{F,h} \| \jump{v_h}_F \|_{L^2(\face)}^2  \right)^\half \text{ in } \Omega
\]
denotes the natural norm used for interior penalty methods. Note that in \eqref{eq:def_ipnorm}, as in the previous section, $\jacobian$ is understood as acting piece wise in each cell, when applied to discontinuous functions. For better readability, we sometimes write $\disnorm{u_h}_{\star, h}:= \left( \int_0^T \disnorm{u_h}_h^2 \wrtdt \right)^\half$.

As corresponding broken Sobolev space, we consider 
$$H^1(\allelements):= \{ \dummyvar \in L^2(\Omega_h;\R^{1+\dime}) : \dummyvar_K \in H^1(K;\R^{1+\dime}) \ \forall \element \in \allelements\}$$  %with norm
and equip it with the norm $\disnorm{\cdot}_h$.
In the analysis, we use the broken Poincaré inequality 
%c.f. \cite[Lemma 2.45]{dolejsi_discontinuous_2015}
\[ \label{eq:poincare1}
\| u_h \|_{L^2(\Omega_h)} \leq c_\textrm{P} \disnorm{ u_h}_h \quad \forall u_h \in H^1(\allelements),
\]
where $c_\textrm{P}>0$ is independent of $h$, cf. \cite{Brenner2003}.
Moreover, we rely on the broken Sobolev inequality 
\[ \label{eq:poincare2}
\| u_h \|_{L^4(\Omega_h)} \leq c_\textrm{S} \disnorm{ u_h}_h \quad \forall u_h \in H^1(\allelements),
\]
where $c_\textrm{S}>0$ is independent of $h$. This inequality is, as the broken Poincaré inequality, just a special case of more general broken Sobolev inequalities. 
%c.f. \cite[Theorem 5.3]{dipietro2012}. 
For a proof of these and other generalized results, we refer the reader to \cite[Theorem 3.7]{Lasis:2003}.
For inequalities \eqref{eq:poincare1}, \eqref{eq:poincare2} to hold, we rely on the assumption that $\alldirfaces_h$ has positive $(d-1)$-dimensional Lebesgue measure, the assumed mesh regularity, as well as $d \in \{2,3\}$.

A key ingredient in the analysis is the following lemma.
\begin{lemma}
$\| \cdot \|_{Z_h}$ is a norm on $Z_h$.
\end{lemma}
\begin{proof}
$ \| \cdot \|_{Z_h} $ is a semi norm, since it is defined by adding different $L^2$-norms together. To show positive definiteness, let $z_h \in Z_h$ be such that $\| z_h \|_{Z_h} = 0$. In particular, this implies $\disnorm{z_h}_h = 0$, which leads to $z_h = 0$, see \cite[Section 6]{dolejsi_discontinuous_2015} for details.
\end{proof}
\subsection{Consistency}
For a consistency result, we define a right hand side with non discrete data, cf. \eqref{eq:l_h},  \eqref{eq:l_h_boundary}, as
\[
\begin{split}
\ell(z_h) := &(\rhs, \numpvtest)_{\numspacetimedomain} + (\M \pv_{0}, \numpvtest(0, \cdot))_{\numspacedomain} 
\\ 
&+ \int_0^\fintime \corr{\Bigg(} - (A^{\textrm{up}}_\norvec u_{D}, \numpvtest)_{\dir}  - \iptheta ( \D u_{D}, \jacobian \numpvtest  \norvec )_{\dir} + \sum_{\face \in \alldirfaces_h} \penpara_{\face, h}( \D {u}_{D}, \numpvtest)_{\face} +
( (p_{N},0)^\transpose, \numpvtest)_{\neu} \\
&\qquad\quad -
 2 (A^{\textrm{up}}_\norvec (0,v_{N})^\transpose, \numpvtest)_{\neu}  - \iptheta ( \D (0,v_{N})^\transpose , \jacobian \numpvtest  \norvec )_{\neu} + \sum_{\face \in \allneufaces_h} \penpara_{\face, h} ( \D( 0,v_{N})^\transpose, \numpvtest)_{\face} \corr{\Bigg)} \wrtdt.
\end{split}
\]
\begin{lemma} \label{le:consistent}
Let $  \hat{u} \in U$ be the solution to the nonlinear equation \eqref{eq:umodel}, and let $u \in U$ denote the solution to the linearized equation \eqref{eq:linearcont} with given $\tild{u} \in U$. Then, it holds
\[ \label{eq:consistent1}
b_h(\hat{u}, \numpvtest) = \ell(\numpvtest) \quad \forall \numpvtest \in Z_h
\] 
and for any $\tild{u}_h \in Z_h$
\[ \label{eq:consistent2}
b_h'(u, \numpvtest; \tild{u}_h) = \ell(\numpvtest) + n_h'(u, z_h; \tild{u}_h - \tild{u})\quad \forall \numpvtest \in Z_h .
\] 
\end{lemma}
\begin{proof}
Due to the regularity of $\hat{u} \in U$, all jump terms in \eqref{eq:full_scheme} involving interior faces vanish. In particular, for any $z_h \in Z_h$ we have
\[  \label{eq:cons1}
\begin{split}
\corr{m_h(\hat{u}, \numpvtest)  
+ a_h(\hat{u}, \numpvtest)  }
= ( L \hat{u} &, \numpvtest)_\numspacetimedomain + ( u_0, z_h(0, \cdot))_\numspacedomain 
\\
&- \int_0^T ( A^{\textrm{up}}_\norvec u_D, z_h )_{\dir} + 2 ( A^{\textrm{up}}_\norvec ( 0,v_N)^\transpose , z_h )_{\neu} \wrtdt ,
\end{split}
\]
and with $z_h=(q_h,w_h)^\transpose$, \corr{using integration by parts, we obtain}
\[ \label{eq:cons2}
\begin{split}
	d_h(\hat{u}, \numpvtest) &= \int_0^T \sum_{\element \in \allelements} (\D \jacobian \hat{u}, \jacobian z_h)_\element +  \bilinspaceipbdr(\hat{u},z_h) \wrtdt
	\\
	&= \int_0^T \sum_{\element \in \allelements} \left(  - ( D \lapl \hat{u}, z_h)_\element +( \D \jacobian \hat{u} \norvec_K, \numpvtest)_{\boundary \element} \right) + \bilinspaceipbdr(\hat{u},z_h)
        \wrtdt \\
	&=  - ( D \lapl \hat{u}, z_h)_\numspacetimedomain + \int_0^T   \sum_{\face \in \allfaces} ( \D \jacobian \hat{u} \norvec_F, \numpvtest)_{\face}  + \bilinspaceipbdr(\hat{u},z_h) \wrtdt \\
	&=  (- D \lapl \hat{u}, z_h)_\numspacetimedomain 
	+ \int_0^T \corr{\Bigg(} \sum_{F \in \alldirfaces_h} \left( - \iptheta (\D u_{D,h}, \jacobian \numpvtest  \norvec_\face )_\face + \penpara_{\face, h} ( \D u_{D,h}, \numpvtest)_F  \right)
	\\
	&
        \qquad
        \qquad
        \qquad
        \qquad
        + ( p_N, q_h)_{\neu} + \parf \sum_{F \in \allneufaces_h} \left( - \iptheta (  v_{N,h}, \jacobian \numvtest  \norvec_\face )_\face + \penpara_{\face, h} ( v_{N,h}, \numvtest)_F \right) \corr{\Bigg)}  \wrtdt,
\end{split}
\]
where we use that $\average{ \jacobian \hat{u}}_F = \jacobian \hat{u}$ and $\jump{\jacobian \hat{u}}_F = 0$ for all $\face \in \allintfaces$,  cf. \cite[Section 4.2.1.1]{dipietro2012}.  

By testing equation \eqref{eq:umodel} locally in every space-time cell with $z_h \in Z_h$, and summing over every space-time cell we obtain
\[
(L \hat{u} - D\lapl \hat{u} + N(\hat{u},\hat{u}), z_h)_{Q_h} = (f, z_h)_{Q_h}.
\]
Combining the previous equation with equations \eqref{eq:cons1} and \eqref{eq:cons2}, we arrive at \eqref{eq:consistent1}. %The second result follows in the same way.
Analogously, testing equation \eqref{eq:linearcont} gives
\[
(L u - D \lapl u + \Non(u,\tild{u}) + \Non(\tild{u},u), z_h)_{Q_h} = (f, z_h)_{Q_h}.
\]
Since equations \eqref{eq:cons1} and \eqref{eq:cons2} also hold when $\hat{u}$ is replaced by ${u}$, we obtain
\[
\begin{split}
b_h'(u, \numpvtest; \tild{u}_h) &= b_h'(u, \numpvtest; \tild{u}_h) + n_h'(u, \numpvtest; \tild{u})  - n_h'(u, \numpvtest; \tild{u}) \\
&= (m_h+a_h+d_h)(u,z_h) + n_h'(u,z_h, \tild{u}_h) +  n_h'(u, \numpvtest; \tild{u})  - n_h'(u, \numpvtest; \tild{u})
\\
&= \ell(\numpvtest) + n_h'(u, z_h; \tild{u}_h - \tild{u}).
\end{split}
\]
\end{proof}
\subsection{Well-posedness} \label{se:well-posedness}
In this section, we establish discrete well-posedness of \eqref{eq:full_scheme} for sufficiently small $\tild{u}_h$. For the dual space of a normed space $X$, we write $X'$. %As an auxiliary step, we first prove the following well-posedness result for the linearized system.
\begin{lemma} \label{lem:linear_wellposed}
	Let $\ell_h \in Z_h'$. There exists $\varrho>0$ such that for all $\tild{u}_h \in Z_h$ with
	\[ \label{eq:smallness_cond}
	 \max_{t\in [0,T]} \disnorm{ \numpvold(t) }_h \leq \varrho \] 
	the variational problem \eqref{eq:full_scheme}
	\[
	b_h'(\numpv, \numpvtest; \numpvold) = \linfull(\numpvtest) \quad \forall \numpvtest \in Z_h
	\] 
has a unique solution $u_h \in Z_h$. Moreover, there exists $c_{\varrho}>0$ independent of $h$ such that the a priori estimate
\[ \label{eq:apriori}
\| u_h \|_{Z_h} \leq c_{\varrho}^{-1} \| \ell_h \|_{Z_h'}
\]
holds.
\end{lemma}
\begin{proof}
Let $u_h \in Z_h$. We show that $b'_h(\numpv, \numpvtest; \numpvold) = 0 \quad \forall \numpvtest \in \numtestspace$ implies $\numpv = 0$, by proving coercivity of the bilinear form $b_h'(\cdot,\cdot; \tild{u}_h) : Z_h \times Z_h \to \R$ with respect to $\| \cdot \|_{Z_h}$.
	
Using relation \eqref{eq:upwindflux_assump} and integration by parts in time we obtain, cf. \cite[Eq. 30]{corallo2022},
\[ \label{eq:coersive_b}
m_h(u_h,u_h) + a_h(u_h,u_h) \geq \half \sum_{j=0}^k \| \jump{u_h}_{j} \|_{L^2(\Omega_h)}^2 + \int_0^T c_A \| A_\norvec \jump{u_h} \|_{L^2(\partial \Omega_h)}^2  \wrtdt .
\]
To estimate the trilinear form $n_h'(\cdot, \cdot; \cdot)$, we combine Hölder's inequality with the broken Sobolev inequality \eqref{eq:poincare2} to obtain
\[  \label{eq:coersive_n}
\begin{split}
	| n_h'(u_h,u_h;\tild{u}_h) | 
	&\leq \int_0^T \sum_{\element \in \allelements}  | ( N(u_h, \tild{u}_h )+ N( \tild{u}_h, u_h), u_h)_{\element} | \wrtdt  \\
	& \leq \mu \int_0^T \sum_{\element \in \allelements} \! \| \jacobian u_h \|_{L^2(\element)} \| u_h \|_{L^4(\element)} \| \tild{u}_h  \|_{L^4(\element)} + \| \jacobian \tild{u}_h \|_{L^2(\element)} \|u_h \|_{L^4(\element)}^2  \wrtdt \\
	& \leq \mu \int_0^T \corr{\Bigg(}
        \left( \sum_{\element \in \allelements}  \| \jacobian u_h \|_{L^2(\element)}^2 \right)^\half \left(\sum_{\element \in \allelements} \| u_h \|_{L^4(\element)}^4 \right)^\frac{1}{4} \left( \sum_{\element \in \allelements} \| \tild{u}_h  \|_{L^4(\element)}^4 \right)^\frac{1}{4} \\
	&\qquad \qquad + \left( \sum_{\element \in \allelements} \| \jacobian \tild{u}_h \|_{L^2(\element)}^2 \right)^\half  \left( \sum_{\element \in \allelements} \|u_h \|_{L^4(\element)}^4 \right)^\half\corr{\Bigg)}  \wrtdt \\
	%& \leq C_n \| \tild{u}_h \|_{L^\infty(0,T;H^1(\domain_h))} \int_0^T \|  u_h\|_{H^1(\domain_h)}^2 \wrtdt \\
	& \leq 2 \mu c_\textrm{S}^2 D_{\max} \int_0^T \disnorm{ u_h }_h^2 \, \disnorm{ \tild{u}_h }_h  \wrtdt \\
	&\leq C_n \left( \max_{[0,T]}\, \disnorm{ \tild{u}_h }_h \right) 
	\disnorm{u_h}_{\star, h}^2  
	%&\leq { 2 \mu c_\textrm{S}^2 D_{max}}  \disnorm{ \tild{u}_h }_{\star,h}   \disnorm{u_h}_{\star, h}^2   \\
	%&\leq C_n \varrho \disnorm{u_h}_{\star, h}^2
	\end{split}
\]
with $\mu:=\max\{ \parc, \pard, \parg, \parh\}>0$, $D_{\max}:=\max\{ \parb^{-\half}, \parf^{-\half}\}$ and $C_n := { 2 \mu c_\textrm{S}^2 D_{\max}   }>0$. 

Combining estimates \eqref{eq:coersive_d}, \eqref{eq:coersive_b}, \eqref{eq:coersive_n}, choosing $\varrho \in (0, c_{\iptheta} C_n^{-1})$, and using the smallness assumption on $\tild{u}_h$ \eqref{eq:smallness_cond} gives
	\[ \label{eq:coercivity}
	\begin{split}
	b'_h(\numpv, \numpv; \numpvold) & \geq m_h(\numpv, \numpv) +  a_h(\numpv, \numpv) + d_h(\numpv, \numpv) - |n'_h( \numpv, \numpv; \numpvold) | \\
 &\geq 
 \half \sum_{j=0}^k \| \jump{u_h}_{j} \|_{L^2(\Omega_h)}^2 + \int_0^T c_{A} \| A_\norvec \jump{u_h} \|_{L^2(\partial \Omega_h)}^2  + ( c_{\iptheta} - \varrho C_n )  \disnorm{ \numpv}^2_h \wrtdt \\
 & \geq c_{\varrho} \| u_h \|_{Z_h}^2
	\end{split}
	\]
with $c_\varrho:= \min\{1, {c_{\iptheta} - \varrho C_n} \}$ independent of $h$. This shows
\corr{coercivity} and the a priori estimate \eqref{eq:apriori} follows by means of standard methods, see e.g. \cite[Lemma 1.4]{dipietro2012}.
\end{proof}
\subsection{Convergence in the discontinuous Galerkin norm}
In this section, we establish convergence of the linear DG scheme \eqref{eq:full_scheme} with respect to $\| \cdot \|_{Z_h}$. To this end, we first introduce relevant projections. 
%For simplicity, we assume in this section that the polynomial degrees are globally constant, i.e., $p=p_R$ and $q=q_R$ for all space-time cells $\spacetimeelement \in \spacetime_h$.
Let $\Pi_{h}: L^2(\Omega;\R^i) \to \polynomialspace_h(\Omega;\R^i) $ with $i \in \N$ denote the spatial $L^2$-projection defined by
\[
(\Pi_{h} u, z_h )_\Omega = (u, z_h )_\Omega \quad \forall z_h \in \polynomialspace_h(\Omega;\R^{i}) 
\]
and $\pi_h : C(0,T;L^2(\Omega;\R^{1+d})) \to Z_h$ denote the space-time \corr{projection} with the following properties
\begin{itemize}
	\item[(i)] $\pi_h u \in Z_h$,
	\item[(ii)] $(\pi_h u )_{j}(x,t_j) = \Pi_h u_j (x,t_j) \text{ for a.e. } x \in \Omega$ and $j \in \{1, \ldots, k\}$ with $i=1+d$,
	\item[(iii)] $\int_{I_j} (\pi_h u - u, z_h)_\Omega \wrtdt = 0 \quad \forall z_h \in Z_h^{-}$ and $j \in \{1, \ldots, k\}$,
\end{itemize}
cf. \cite[Section 6.1.4]{dolejsi_discontinuous_2015}, where existence and uniqueness among other properties of $\pi_h$ are shown. We denote by $Z^-_h$ the space obtained when decreasing the polynomial time degrees by one in each space-time cell. Due to this property, we assume for the rest of the paper that, the time degrees of each cell are at least one.
The following result is the basis for an estimate of the numerical error in terms of projection errors.
\begin{lemma} \label{le:convergence_1}
	Let $u \in U$ be the solution to the continuous linear system \eqref{eq:linearcont} with  $\tild{u} \in U$ and $u_h \in Z_h$ the solution to the discrete linear problem \eqref{eq:full_scheme} with $\tild{u}_h \in Z_h$ such that \eqref{eq:smallness_cond} holds with $\varrho>0$ as in Lemma \ref{lem:linear_wellposed}. Then
	\[
	\begin{split}
	  \| u_h - \pi_h u \|_{Z_h}^2 \leq  \ &C_Z \left( \sum_{j=1}^k \! \| \M^\half ( \pi_h u - u  )_{j}(t_j) \|^2_{L^2(\Omega_h)}
          + \int_0^T \corr{\Big(}
          \| \pi_h u - u \|_{L^2(\partial \Omega_h)}^2
          + \disnorm{ \pi_h u - u }_{ h}^2\corr{\Big)} \wrtdt \right) \\  
		&+ C_R \left( \| \ell_h - \ell \|_{Z_h'}^2 + \disnorm{ \tild{u}_h - \tild{u} }^2_{\star, h}    \right) 
	\end{split}
	\]
	with $C_Z,C_R>0$ independent of $h$.
\end{lemma}
\begin{proof}
	By consistency, see Lemma \ref{le:consistent} equation \eqref{eq:consistent2}, it holds for all $z_h \in Z_h$
	\[
	b_h'( u_h - u, z_h; \tild{u}_h ) = b_h'( u_h , z_h; \tild{u}_h ) - b_h'( u, z_h; \tild{u}_h )  = \ell_h(z_h) - \ell(z_h) - n_h'(u, z_h; \tild{u}_h - \tild{u}) ,
	\]
	which is equivalent to
	\[
	b_h'( u_h - \pi_h u, z_h; \tild{u}_h ) = - b_h'( \pi_h u  - u, z_h ; \tild{u}_h ) +  \ell_h(z_h) - \ell(z_h) - n_h'(u, z_h; \tild{u}_h - \tild{u}) .
	\]
	Choosing $z_h = u_h - \pi_h u \in Z_h $, defining remainder terms $\Rem := | (\ell_h- \ell)( u_h - \pi_h u) | + | n_h'(u,  u_h - \pi_h u; \tild{u}_h - \tild{u}) |$, and
	by coercivity of $b_h'$ derived in the proof of Lemma \ref{lem:linear_wellposed}, see equation \eqref{eq:coercivity}, we obtain
	\[ \label{eq:error_estimate1}
	c_\varrho \| u_h - \pi_h u \|_{Z_h}^2 \leq \vert b'_h( \pi_h u  - u, u_h - \pi_h u ; \tild{u}_h ) \vert +  \Rem.
	\]
	For the rest of the proof, we estimate the right hand side of \eqref{eq:error_estimate1} further.
	Starting with $m_h$, we have for all $z_h \in Z_h$ using partial integration in time, property (iii) of $\pi_h$, Hölder's inequality, and Young's inequality
	\[
	\begin{split}
		\vert m_h (\pi_h u - u, z_h ) \vert &= \left\vert \sum_{j=1}^k \int_{t_{j-1}}^{t_j} (  \M \dt (\pi_h u - u) , z_h )_{\Omega_h} \wrtdt 
		+ (M \jump{\pi_h u - u}_{j-1} , z_{h,j}(t_{j-1}) )_{\Omega_h} \right\vert
		\\
		&= \left\vert \sum_{j=1}^k - \int_{t_{j-1}}^{t_j} (  \M (\pi_h u - u) , \dt z_h )_{\Omega_h} \wrtdt 
		- (M (\pi_h u - u )_{j}(t_j) , \jump{z_h}_{j}  )_{\Omega_h} \right\vert
		\\
		&= \left\vert \sum_{j=1}^k (M^\half (\pi_h u - u )_{j}(t_j) , M^\half \jump{z_h}_{j}  )_{\Omega_h}\right\vert 
		\\
		%&\leq \sum_{j=1}^k \| M^\half (\pi_h u - u )_{j}(t_j) \|_{L^2(\Omega_h)} \| M^\half \jump{z_h}_{j} \|_{L^2(\Omega_h)} 
		%\\
		&\leq \sqrt{2} \left( \sum_{j=1}^k \| M^\half (\pi_h u - u )_{j}(t_j) \|^2_{L^2(\Omega_h)} \right)^\half \left( \half \sum_{j=1}^k \| M^\half \jump{z_h}_{j} \|^2_{L^2(\Omega_h)} \right)^\half 
		\\
		&\leq \frac{2 \sqrt{2}}{c_\varrho} \sum_{j=1}^k \| M^\half (\pi_h u - u )_{j}(t_j) \|^2_{L^2(\Omega_h)} + \frac{c_\varrho}{8}  \| z_h \|^2_{Z_h} ,
	\end{split}
	\]
	cf. \cite[Section 6.1]{dolejsi_discontinuous_2015}.
	For $a_h$, we have for all $z_h \in Z_h$ using Hölder's inequality, Lemma \ref{le:upwind_symetric},  discrete Poincaré inequality \eqref{eq:poincare1}, and relation \eqref{eq:upwindflux_assump}
	\[
	\begin{split}
	  \vert  \bilinspacesf( \pi_h u - u , \numpvtest)  \vert 
		&\leq   \int_0^T \!\!\sum_{\element \in \allelements} \! \left(   \| \A (\pi_h u - u) \|_{L^2(K)} \|\numpvtest \|_{L^2(\element)} + \sum_{\face \in \face_\element} \| { \pi_h u - u } \|_{L^2(\face)} \| \Aup \jump{\numpvtest}_{\face, \element} \|_{L^2(\face)}  \right) \wrtdt
		%	\\
		%	&\leq   \int_0^T    \| \A (\pi_h u - u) \|_{L^2(\Omega_h)} \| \numpvtest \|_{L^2(\Omega_h)} + \left( \sum_{\element \in \allelements} \sum_{\face \in \face_\element} \| \pi_h u - u\|_{L^2(\face)}^2 \right)^\half \left( \sum_{\element \in \allelements} \sum_{\face \in \face_\element}\| \Aup \jump{ \numpvtest}_{\face, \element} \|_{L^2(\face)}^2 \right)^\half  \wrtdt
		\\
		&\leq \int_0^T   \|  \jacobian (\pi_h u - u) \|_{L^2(\Omega_h)} \|\numpvtest \|_{L^2(\Omega_h)} + 2 C_A \| \pi_h u - u \|_{L^2(\partial \Omega_h)} 
		\| A_\norvec \jump{\numpvtest} \|_{L^2(\partial \Omega_h)}  \wrtdt \\
		&\leq \int_0^T  \frac{2c_\textrm{P} D_{\max} }{c_\varrho c_{\iptheta}}  \disnorm{ {\pi_h u - u} }_h^2 + \frac{c_\varrho c_{\iptheta}}{8}  \disnorm{z_h}_h^2 + \frac{4 C_A}{c_\varrho c_A} \| \pi_h u - u \|_{L^2(\partial \Omega_h)}^2 \wrtdt + \frac{c_\varrho}{8}  \| z_h \|^2_{Z_h}    .
	\end{split}
	\]
	%with $D_{\max}:=\max\{ \parb^{-\half}, \parf^{-\half}\}$.
	For $d_h$, we have for all $z_h \in Z_h$ using Hölder's inequality and Young's inequality
	\[
	\begin{split}
		| \bilinspaceip(\pi_h u - u , \numpvtest) | &\leq \int_0^T \| \D^\half \jacobian (\pi_h u - u) \|_{L^2(\Omega_h)} \| D^\half \jacobian \numpvtest \|_{L^2(\numspacedomain)} + | \bilinspaceipbdr(\pi_h u - u, \numpvtest) | \wrtdt  
		\\
		&\leq \int_0^T \disnorm{\pi_h u - u}_h  \disnorm{z_h}_h + | \bilinspaceipbdr(\pi_h u - u, \numpvtest) | \wrtdt
		\\
		&\leq\int_0^T  \frac{2 }{ c_\varrho c_{\iptheta}} \disnorm{ \pi_h u - u}_h^2 + \frac{c_\varrho c_{\iptheta}}{8}  \disnorm{z_h}_h^2 + | \bilinspaceipbdr(\pi_h u - u, \numpvtest) | \wrtdt,
	\end{split}
	\]
	while for $C_\textrm{ip}>0$ independent of $h$
	\[
	\begin{split}
		| \bilinspaceipbdr(\pi_h u - u, \numpvtest) | \leq C_\textrm{ip} D_{\max}^2 \disnorm{\pi_h u - u }_{h} \disnorm{ z_h }_h \leq \frac{2 C_\textrm{ip} D_{\max}^2}{c_\varrho c_{\iptheta}}  \disnorm{\pi_h u - u }_{h}^2 + \frac{c_\varrho  c_{\iptheta}}{8 } \disnorm{ z_h }_h^2
	\end{split},
	\]
	which can be proven by combining \cite[Lemma 4.16]{dipietro2012} and  \cite[Lemma 4.20]{dipietro2012}.
	Finally, for $n'_h$, we have for all $z_h \in Z_h$ using estimates as in \eqref{eq:coersive_n}, $\max_{[0,T]} \disnorm{\tild{u}_h }_{h} \leq \varrho$, and Young's inequality
	\[  \label{eq:test1}
	%\begin{split}
	| n'_h(\pi_h u - u,z_h;\tild{u}_h) | 
	%&\leq \mu \int_0^T \sum_{\element \in \allelements} \| \jacobian (\pi_h u - u) \|_{L^2(\element)} \| z_h \|_{L^4(\element)} \| \tild{u}_h  \|_{L^4(\element)} +  \| \pi_h u - u\|_{L^4(\element)} \| z_h \|_{L^4(\element)} \| \jacobian \tild{u}_h \|_{L^2(\element)} \wrtdt \\
	%& \leq C_n \| \tild{u}_h \|_{L^\infty(0,T;H^1(\domain_h))} \int_0^T \|  u_h\|_{H^1(\domain_h)}^2 \wrtdt \\
	%& \leq \mu \int_0^T  c_S^2 \|{ \jacobian (\pi_h u - u)}\|_{L^2(\Omega_h)} \disnorm{ z_h }_h  \disnorm{ \tild{u}_h }_h + c_S D_{\max} \| \pi_h u - u \|_{L^4(\Omega_h)} \disnorm{ z_h }_h  \disnorm{ \tild{u}_h }_h   \wrtdt \\
	%& \leq \mu \varrho \left(  c_S^2 C_d \|{ \jacobian (\pi_h u - u)}\|_{L^2(0,T;L^2(\Omega_h)}  + c_S D_{\max} \| \pi_h u - u \|_{L^2(0,T;L^4(\Omega_h)} \right) \left(\int_0^T \disnorm{ z_h }_h^2 \wrtdt \right)^\half \\
	\leq \int_0^T \frac{2 C_n \varrho}{c_\varrho  c_{\iptheta}} \disnorm{\pi_h u - u}^2_h + \frac{c_\varrho  c_{\iptheta}}{8}  \disnorm{ z_h }_h^2 \wrtdt . 
	%\end{split}
	\]
	Choosing $z_h = u_h - \pi_h u \in Z_h$, and absorbing the $\| \cdot \|_{Z_h}$ terms into the left hand side of \eqref{eq:error_estimate1}, we obtain
	\[
	\begin{split}
	  \frac{c_\varrho}{4} \| u_h - \pi_h u \|_{Z_h}^2 
		&\leq C_b \left( \sum_{j=1}^k \! \| \M^\half ( \pi_h u - u  )_{j}(t_j) \|^2_{L^2(\Omega_h)} +  \int_0^T \| \pi_h u - u \|_{L^2(\partial \Omega_h)}^2 + \disnorm{ \pi_h u - u }_h^2 \wrtdt \right)  +\Rem
	\end{split}
	\]
	with $C_b:= 2 c_{\varrho}^{-1} \max\{ \sqrt{2}, \frac{c_\textrm{P} D_{\max}}{c_{\iptheta}}, \frac{2 C_A}{c_A}, \frac{1}{c_{\iptheta}}, \frac{C_\textrm{ip} D_{\max}^2}{c_{\iptheta}}, \frac{C_n \varrho}{c_{\iptheta}} \} >0$.
	Since
	\[
	\begin{split}
		\Rem  
		%&\leq \| \pi_h u - u \|_{Z_h} \| \ell_h - \ell \|_{Z_h'} + C_n \| u \|_{U} \|u_h - \pi_h u \|_{Z_h} \| \tild{u}_h - \tild{u} \|_{Z_h} \\
		\leq 
		\frac{c_\varrho}{16} \| \pi_h u - u \|_{Z_h}^2 + \frac{4}{c_{\varrho}} \| \ell_h - \ell \|^2_{Z_h'} + \frac{c_\varrho}{16} \|u_h - \pi_h u \|_{Z_h}^2 + \frac{4  C_n c_{\iptheta}  \max_{[0,T]}\disnorm{ u }_{h} }{c_{\varrho} }  \disnorm{ \tild{u}_h - \tild{u} }_{\star, h}^2,
	\end{split}
	\]
	using Young's inequality and \eqref{eq:coersive_n}, we obtain the result with
        $C_R:= \frac{32}{c_\varrho^2} \max \{ 1, {C_n c_{\iptheta} \max_{[0,T]}\disnorm{ u }_{h} } \}$
        and
        $C_Z:=\frac{8 C_b}{c_{\varrho}}$.
\end{proof}
\begin{remark}
	For matters of readability, we use $\disnorm{\cdot}_h$ to estimate $\| \jacobian (\pi_h u - u)\|_{L^2}$ and $\| \pi_h u - u \|_{L^4}$.  
	In principle, one does not have to do this and one would obtain slightly better constants $C_P, C_{p'}$ in the resulting error estimate cf. Theorem \ref{th:convergence}. However, the overall scaling with $h$ in the error estimate is not affected.  
\end{remark}
\begin{remark} \label{re:C_U}
	Note that $\max_{t\in [0,T]} \disnorm{u(t)}_h$ is uniformly bounded in $h$ for $u \in U$. This follows from the Sobolev embedding
	$ L^\infty(0,T;H^1(\Omega)) \embed L^2(0,T;H^2(\Omega)) \cap H^1(0,T;L^2(\Omega)) $, see \cite[Sect.~5.9, Thm.~4]{Evans:2010} and the fact that $\disnorm{u}_h \leq C_U \| u \|_{H^1(\Omega)}$ for some $C_U>0$ independent of $h$.
\end{remark}
To complete the convergence analysis, we estimate the projection errors of $\pi_h$ in terms of $h$. For this, we define the following semi-norms
\[
\begin{split}
	| u |_{H^s(\Omega)} &= \left( \int_\Omega \sum_s | \partial_x^s u |^2 \, \mathrm{d}x \right)^{\frac{1}{2}} , \\
	|u |_{H^p(0,T;H^s(\Omega))} &= \left( \int_0^T| \dt^p u |^2_{H^s(\Omega)} \wrtdt \right)^{\frac{1}{2}} ,\\
	|u |_{C(0,T;H^s(\Omega))} &= \max_{[0,T]} | u |_{H^s(\Omega)} , \\
	| u ||_{H^{p}(0,T;H^1(\Omega))} &= \left( |u|^2_{H^p(0,T;L^2(\Omega))} + |u|^2_{H^p(0,T;H^1(\Omega))}  \right)^\half,
\end{split}
\]
where $\partial_x^s$ stands for spatial partial derivatives of order $s \in \N$. The next lemma mainly uses results from \cite[Section 6]{dolejsi_discontinuous_2015}. In the following, we set $ \vartheta = 1$ if for the Dirichlet data with global minimal time degree~$p$ we have
\[ \label{eq:dirdata}
u_\mathrm{D}(x,t ) = \sum_{j=0}^p \Psi_j(x) t^j 
\]  %norm is also semi norm
with $\Psi_j \in H^{s - \half}(\dir;\R^{1+d})$, $s \geq 1$ and $\vartheta=0$ otherwise.

\clearpage

\begin{lemma} \label{le:convergence_2}
	Let 
	\[ \label{eq:error_regularity}
	 V := H^{p+1}(0,T;H^1(\Omega;\R^{1+d})) \cap W^{p+1, \infty}(0,T;L^2(\Omega;\R^{1+d}) ) \cap C([0,T];H^s(\Omega;\R^{1+d}) )
	\]
	with $p=\min_{R \in \spacetime_h}p_R$ the global minimal time degree, $q=\min_{R \in \spacetime_h}q_R$ the global minimal space degree, $\nu=\min\{q+1, s\}$ with $s \geq 1$, $\omega \in [1,\infty)$ the constant from \eqref{eq:mesh_balance} and $\vartheta \in \{0 , 1\}$ the constant determined by \eqref{eq:dirdata}. Then, the following projection error estimates hold for $u \in V$
%	\[ 
	\begin{align}
	\label{eq:proj_error_timejump}
	\sum_{j=0}^k \| \M^{\frac{1}{2}} \jump{\pi_h u - u}_{j} \|_{L^2(\Omega_h)}^2
        &\leq C_1 \left( h^{2 \nu - \omega} | u |^2_{C(0,T;H^\nu(\Omega))} + h^{2( p + 1)- \omega} | u |^2_{W^{p+1, \infty}(0,T;L^2(\Omega))}  \right)
        \,,
%	\]
	%
\\
	%\[ 
	\label{eq:proj_error_timestep}
	\sum_{j=1}^k \! \| M^\half ( \pi_h u - u  )_{j}(t_j) \|^2_{L^2(\Omega_h)}
        &\leq C_2 h^{2 \nu - \omega } | u  |_{C(0,T;H^\nu(\Omega))}^2
        \,,        
	%\]
	\\
	%and
	%\[ 
	\label{eq:proj_error_disnorm}
	\int_0^T \disnorm{ \pi_h u - u}_h^2 \wrtdt 
        &\leq C_3 \left( h^{2(\nu-1)} | u |^2_{C(0,T;H^\nu(\Omega))} + h^{2 (p + \vartheta)} | u ||^2_{H^{p+1}(0,T;H^1(\Omega))}  \right)
        \,,
	%\]
	\\
	%
%	\[ 
\label{eq:proj_error_boundary}
\int_0^T \| \pi_h u - u \|_{L^2(\partial \Omega_h)}^2 \wrtdt
&\leq C_4 \left(  h^{2 (\nu - \half)} |u|^2_{C(0,T;H^\nu(\Omega))} + h^{2(p+\half)} |u|^2_{H^{p+1}(0,T;L^2(\Omega))} \right) 
        \,,
%	\]
	%
	\\
%	\[
	 \label{eq:proj_error_spacejump}
	 \int_0^T c_{\A} \| A_\norvec \jump{\pi_h u - u} \|_{L^2(\partial \Omega_h)}^2
         &\leq  C_5 \left(  h^{2 (\nu - \half)} |u|^2_{C(0,T;H^\nu(\Omega))} + h^{2(p+\half)} |u|^2_{H^{p+1}(0,T;L^2(\Omega))} \right) 
        \,,
%	\]
		\end{align}
	%	\]
	%and
	where $C_1$ to $C_5$ are positive constants independent of $h$.
\end{lemma}
\begin{proof}
  From \cite[Lemma 6.22, Eq. (6.136)]{dolejsi_discontinuous_2015}
  we \corr{get, using \eqref{eq:mesh_balance} and taking $\sup_{[0,T]}$,}
	\[ \label{eq:error-projL2}
	\| \pi_h u - u \|_{C(0,T;L^2(\Omega_h))}^2 \leq 2 C_{\pi}^2 \left( h^{2\nu} |u|^2_{C(0,T;H^\nu(\Omega))} + c_r^{-2(p+1)} h^{2(p+1)} | u|^2_{W^{p+1, \infty}(0,T;L^2(\Omega))} \right)	
	\]
	with $C_{\pi}>0$ independent of $h$.
	To derive projection error estimate \eqref{eq:proj_error_timejump}, we use relation \eqref{eq:mesh_balance} together with $(k+1) \leq C_{s} \tau^{-1}$ with $C_s>0$, which holds due to shape regularity, so that %\textcolor{red}{WRONG}
	\[
	\sum_{j=0}^k \| \jump{ \pi_h u- u}_{j} \|_{L^2(\Omega_h)}^2 \leq 2 (k+1) \| \pi_h u - u   \|_{C(0,T;L^2(\Omega) )}^2 \leq 2 C_s c_r h^{-\omega} \|  \pi_h u - u \|_{C(0,T;L^2(\Omega))}^2.
	\]
	Combining this with \eqref{eq:error-projL2}, we obtain \eqref{eq:proj_error_timejump} with $C_1 = 4 C_s c_r C_\pi^2 \max\{  1, c_r^{-2(p+1)} \} \max\{\alpha^{-\half}, \varepsilon^{-\half} \}.$
	
	According to \cite[Lemma 6.17, Eq. 107]{dolejsi_discontinuous_2015}, we have for $j=1,\dots, k$
	\[
	\| ( \pi_h u - u  )_j(t_j) \|^2_{L^2(\Omega_h)} \leq C_l h^{ 2 \nu } | u_{j}(t_j)  |_{H^\nu(\Omega)}^2 
	\]
	with $C_{l}>0$ independent of $h$. Summing over $j$, taking supremum over $j$, 
	%, multiplying by $M^\half$ 
	using $k \leq C_s \tau^{-1}$, and mesh balance \eqref{eq:mesh_balance} gives  \eqref{eq:proj_error_timestep} with $C_2= C_l C_s c_r$.
	
	Due to \cite[Lemma 6.22, Eq. 134]{dolejsi_discontinuous_2015}, we have for $j=1, \dots, k$
	\[
	\int_{I_j} \disnorm{ u - \pi_h u}_h^2 \wrtdt \leq C_{\vert\!\vert\!\vert\!} \left( h^{2(\nu-1)} | u |^2_{L^2(I_j;H^\nu(\Omega))} + {\tau_j}^{2 ( p + \vartheta)} | u ||^2_{H^{p+1}(I_j;H^1(\Omega))}  \right)
	\]
	with $C_{\vert\!\vert\!\vert\!}>0$ independent of $h$.
	From there, equation \eqref{eq:proj_error_disnorm} follows by summing over $j$, using \eqref{eq:mesh_balance} $c_r \tau \leq h$, and 
%	\[ \label{eq:time-L2-Linf}
$	|u|^2_{L^2(0,T;H^\nu(\Omega))} \leq T |u|^2_{C(0,T;H^\nu(\Omega))}$
%	\]
	with $C_3 = C_{\vert\!\vert\!\vert\!} T \max\{1, c_r^{-2(p + \vartheta)}\}$.
	
	By the discrete trace inequality, cf. \cite[Remark 1.47]{dipietro2012}, we have
	\[
	\| \pi_h u - u \|_{L^2(\boundary \Omega_h)}^2 
	\leq \sum_{\element \in \allelements} \| u - \pi_h u \|^2_{L^2(\boundary \element)} 
	\leq C_\textrm{tr} h^{-1} \sum_{\element \in \allelements} \| u - \pi_h u \|^2_{L^2(\element)} 
	\]
	with $C_\textrm{tr}>0$ independent of $h$.
	Combining this with relation \eqref{eq:error-projL2} gives
	\[ 
	\int_0^T \| u - \pi_h u \|^2_{L^2(\partial \Omega)} \wrtdt \leq 2 h^{-1} C_\textrm{tr} C_{\pi}^2 \left(  h^{2 \nu} |u|^2_{C(0,T;H^{\newcorr{\nu}})} + h^{2(p+1)} |u|^2_{W^{p+1,\infty}(0,T;L^2(\Omega))} \right) ,
	\]
	which yields equation \eqref{eq:proj_error_boundary} with $C_4=2 C_\textrm{tr} C_\pi^2 \max\{ 1, c_r^{-2(p+1)} \} $.
	
	Similarly, equation \eqref{eq:proj_error_spacejump} with $C_5 = 2 C_4$ can be obtained, since 
	\[
	\| A_\norvec \jump{\pi_h u - u} \|^2_{L^2(\partial \Omega_h)} \leq 2 \| \pi_h u - u \|^2_{L^2(\partial \Omega)} .
	\]	
\end{proof}
The following theorem provides a complete error estimate for the linearized system in terms of $h$.
\begin{theorem} \label{th:convergence}
	Let $u \in V$ with $V$ from \eqref{eq:error_regularity} be a solution to the continuous linear system \eqref{eq:linearcont} with $\tild{u} \in U$ and let $u_h\in Z_h$ be a solution to the discrete linear system \eqref{eq:full_scheme} with $\tild{u}_h \in Z_h$ such that \eqref{eq:smallness_cond} holds with $\varrho>0$ as in Lemma \ref{lem:linear_wellposed}.
	
	Furthermore, 
	let $p=\min_{R \in \spacetime_h}p_R$ be the global minimal time degree, $q=\min_{R \in \spacetime_h}q_R$ be the global minimal space degree, $\omega \in [1,\infty)$ the constant from \eqref{eq:mesh_balance}, $\vartheta \in \{0 , 1\}$ the constant determined by \eqref{eq:dirdata}, and $\xi=\min\{p + \vartheta, p + 1 - \frac{\omega}{2}, \nu - \frac{\omega}{2}, \nu - 1\}$ with $\nu=\min\{q+1, s\}$, $s \geq 2$. Then
	\[
	\begin{split}
		\| u_h - u \|_{Z_h} \leq  C_{P}  h^{\xi}  + 
		C_{Q} \left( \| \ell_h - \ell \|_{Z_h'} + \disnorm{ \tild{u}_h - \tild{u} }_{\star, h}     \right)
	\end{split}
	\]
	with $C_{P},C_{Q}>0$ independent of $h$.
\end{theorem}
\begin{proof}
	By the triangle inequality, we have
	\[
	\| u_h - u \|_{Z_h} \leq \| u_h - \pi_h u \|_{Z_h} + \| \pi_h u - u \|_{Z_h}.
	\]
	From Lemma \ref{le:convergence_2} we obtain, when taking square roots,
	\[
	\| \pi_h u - u \|_{Z_h} \leq C_6  h^{\xi} 
	\]
	using the projection estimates \eqref{eq:proj_error_timejump}, \eqref{eq:proj_error_disnorm}, \eqref{eq:proj_error_spacejump}.
	%with \eqref{eq:time-L2-Linf}.
	Combining Lemma \ref{le:convergence_1} with the projection estimates \eqref{eq:proj_error_timestep}, \eqref{eq:proj_error_disnorm}, \eqref{eq:proj_error_boundary} from Lemma \ref{le:convergence_2} we obtain, when taking square roots,
	\[
	\| \pi_h u - u_h \|_{Z_h} \leq C_7  h^{\xi} + C_8 \left( \| \ell_h - \ell \|_{Z_h'} +  \disnorm{ \tild{u}_h - \tild{u} }_{\star, h}     \right),
	\]
	yielding the error estimate for the linear problem, since the constants $C_6,C_7,C_8>0$ are only depending on $\| u \|_V, C_Z, C_R$ and $C_i$ with $i =1, \ldots, 5$.
\end{proof}
\begin{remark}
	We note that $U \subset V$ for $p=0$. If $\vartheta= \omega = 1$, Theorem \ref{th:convergence} yields convergence in $h$
	if 
	$$u \in H^1(0,T;H^1(\Omega;\R^{1+d})) \cap W^{1,\infty}(0,T;L^2(\Omega;\R^{1+d})) \cap C(0,T;H^2(\Omega;\R^{1+d}))$$
	choosing $p=0$, $q=1$ and $s=2$.
A convergence analysis for less regular solutions is not contained in this paper.
	%This is archived if $f \in H^1(0,T;H^1)$ (for the heat equation with regular enough domain). Make this into a remark?
\end{remark}
\section{Nonlinear system}\label{se:nonlinear}
We are now in a position to apply a general existence and convergence result for Newton's method to the nonlinear discrete problem \eqref{eq:full_scheme_nonlinear}, namely, a generalized Newton-Kantorovich Theorem, see \cite[Theorem 2.1]{Hernandez2001}. We restate the theorem for the reader's convenience. 
\begin{lemma}\label{NewtonKantorowich} 
	Let $X$ and $Y$ be Banach spaces and $\G:\mathcal{D}(\G)\subseteq X \to Y$. Suppose that on an open convex set $D_*\subseteq \mathcal{D}(\G)$, $\G$ is Fr\'{e}chet differentiable and $\G'$ is $\chi$-Hölder continuous in $D_*$ with  exponent $\chi \in (0,1]$ and constant $\lambda>0$, i.e.,
	\[
	\|\G'_{u_1}-\G'_{u_2}\|_{\linearspace(X,Y)}\leq \lambda \|u_1-u_2\|_{X}^\chi , \quad u_1, \, u_2\in D_*.
	\]
	For some $u_*\in D_*$, assume $\Gamma_*:=\corr{(\G'_{u_*})^{-1}}$ is defined on all of $\, Y$ and  
	\begin{equation}\label{betaK}
		\upsilon = \kappa \lambda \phi^\chi \leq \chi_0 \ \text{ with } \kappa = \|\Gamma_*\|_{\linearspace(Y,X)}, \quad \phi = \|\Gamma_*\G(u_*) \|_X,
	\end{equation}
	where $\chi_0 \in (0, \half]$ is the unique root of $ w \mapsto (1+ \chi)^\chi(1-w)^{1 + \chi} - w^\chi$. 
	Set 
	\begin{equation}\label{rstart_rstarstar}
		\varrho= \frac{ \phi (1+\chi)(1-\upsilon)}{(1+\chi)-(2 + \chi) \upsilon}, \qquad \tild{\varrho} = \phi \upsilon^{- \frac{1}{\chi}}
	\end{equation}
	and suppose that $B^X_{\varrho}(u_*)\subseteq D_*$. 
	Then, the Newton iterates 
	\begin{equation}\label{Newton}
		u^{(n+1)}=u^{(n)}-{\G'^{-1}_{u^{(n)}}}\G(u^{(n)}), \, n=0,1,2,\ldots 
	\end{equation}
	starting at $u^{(0)}:=u_*$ 
	are well-defined, lie in $B^X_{\varrho}(u_*)$ and converge to a solution of $ \G(u)=0$, which is unique in $B^X_{\tild{\varrho}}(u_*)\cap D_*$. Moreover, if strict inequality holds in \eqref{betaK}, the order of convergence is at least $1+\chi$.
\end{lemma}
Building upon the results of the previous section, we derive the following local well-posedness and convergence result for the nonlinear problem \eqref{eq:full_scheme_nonlinear}. In the proof, we use that $\| u \|_{Z_h} \leq C_U \| u \|_U$ for $u \in U$ and some $C_U>0$ independent of $h$ cf. Remark \ref{re:C_U}.
\begin{theorem} \label{th:nonlinear_wellposed}
	Let $u \in V$ with $V$ from \eqref{eq:error_regularity} be a solution to the continuous nonlinear system \eqref{eq:umodel} and assume that 
	\[ \label{eq:smallness_ass}
	\| u \|_U \leq \varrho_U, \qquad  \| \ell_h \|_{Z'_h} \leq \varrho_\ell ,
	\] 
	where $\varrho_U>0$ and $\varrho_\ell>0$ are assumed to be sufficiently small.
	
	Furthermore, let $p=\min_{R \in \spacetime_h} p_R$ be the minimal global time degree, $q=\min_{R \in \spacetime_h}q_R$ be the minimal global space degree, $\omega \in [1,\infty)$ the constant from \eqref{eq:mesh_balance}, $\vartheta \in \{0 , 1\}$ the constant determined by \eqref{eq:dirdata}, $\xi=\min\{p + \vartheta, p + 1 - \frac{\omega}{2}, \nu - \frac{\omega}{2}, \nu - 1\}$ with $\nu=\min\{q+1, s\}$, $s \geq 2$, and $\chi \in (0,1]$ such that $(1- \chi) \min\{ p + \vartheta, p+1-\frac{\omega}{2} \} \geq \frac{\omega}{2}$. We also assume that
	$ \min\{ p + \vartheta - \omega, p + 1 - \frac{3 \omega}{2}\} > 0$ and $\nu > \omega$.
	 For $\bar{h}>0$ define
	\[ D_{\bar{h},*} := \{ u_h \in Z_h : ( \bar{h}^{- \min\{ p + \vartheta, p+1-\frac{\omega}{2} \} } + \bar{h}^{- \min\{\nu - \frac{\omega}{2}, \nu - 1\} } ) \| u_h - u \|_{Z_h} 
	%+ h^{\half} \| u - u_h \|_{L^2(0,T;L^2(\partial \Omega_h))} + h^{\frac{\omega}{2}} \| \dt (u - u_h) \|_{L^2(0,T;L^2(\Omega))} 
	%< C_* \bar{h}^\xi \} 
	< C_* \}
	\]
	with $C_*>0$ independent of $\bar{h}$.
	Then, there exists a $\tild{h} < \bar{h}$ such that for all $h < \tild{h}$
		\[
		b_h(\numpv, \numpvtest) = \linfull(\numpvtest) \quad \forall \numpvtest \in Z_h
		\] 
		has a unique solution $u_h \in Z_h$
		%for all $$ \| f_h \|_{Q_h} + \| u_{0,h} \|_{\domain_h} + \| u_{D,h} \| + \| u_{N,h} \| \leq r .$$
		and ($1+\chi$)-order of convergence in Algorithm \ref{algo:newton} with initial guess $u_{h,*} \in D_{h,*}$ holds.	
	Moreover, we have the following error estimate 
	\[
	\| u_h - u \|_{Z_h} 
	%+
	%h^\half \| u - u_h \|_{L^2(0,T;L^2(\partial \Omega_h))} 
	%+ h^\omega  \| \dt (u - u_h) \|_{L^2(0,T;L^2(\Omega))} 
	\leq C_* h^\xi.
	\]
\end{theorem}
\begin{proof}
Let $h \leq \bar{h} < 1$. We apply Lemma \ref{NewtonKantorowich} to 
\[
\G_h : Z_h \to Z_h' \text{ defined by }  \langle \G_h(u_h), \cdot  \rangle_{Z_h'} = b_h(u_h, \cdot) -  \ell_h(\cdot)  
\] 
%Let $\varrho:= \frac{c_{\iptheta}}{2 C_n}>0$ and 
using the open, convex set $D_{\bar{h}, *} \subset Z_h$. Note that due to Lemma \ref{le:convergence_2}, $D_{*,h}$ is non-empty, when choosing $C_*$ suitably, since then $\pi_h u \in D_{*,h}$. This follows from the projection error estimates, trace inequality and inverse inequality, see equations \eqref{eq:bounaryL2}, \eqref{eq:timederi} below.

The Fr\'{e}chet derivative of $G_h$ at $\tild{u}$ evaluates to $ b'_h(u_h,z_h;\tild{u})$, cf. \cite[Lemma 3.2]{QuadraticWave}. First, we show its well-posedness on $D_{*,\bar{h}}$ using Lemma \ref{lem:linear_wellposed}. 
%equation \eqref{eq:coersive_n}, we derive for $\hat{u}_h \in Z_h$ and 
For $u_{h,*}$ in $D_{h,*}$ we have 
\[ \label{eq:coersKant}
\begin{split}
%	| n_h'(\hat{u}_h,\hat{u}_h;u_{h,*}) | 
%	&\leq | n_h'(\hat{u}_h,\hat{u}_h; u_{h,*} - u) | + | n_h'(\hat{u}_h,\hat{u}_h;  u ) | \\
\max_{[0,T]} \disnorm{   u_{h,*}}_h
	&\leq   \max_{[0,T]} \disnorm{ u_{h,*} - u}_h +  \max_{[0,T]} \disnorm{  u}_h    \\
	&\leq   C_\textrm{inv} \tau^{-\half} \disnorm{u_{h,*} - u}_{\star,h} + C_U \| u \|_U\\
	&\leq C_\textrm{inv} h^{-\frac{\omega}{2}} \sqrt{c_r} c_\iptheta^{-1} \|u_{h,*} - u\|_{Z_h} + C_U \varrho_U   \\
	&\leq  C_\textrm{inv} \sqrt{c_r} c_\iptheta^{-1} C_* \bar{h}^{ \min\{ p + \vartheta - \frac{\omega}{2}, p+1- \omega \}} + C_U \varrho_U 
\end{split}
\]
where $C_U>0$ from Remark \ref{re:C_U} is used, $C_\textrm{inv}>0$ comes from an inverse inequality to estimate the $L^\infty$-norm with respect to the $L^2$-norm in time, mesh balance \eqref{eq:mesh_balance} and $\|  u_{h,*} - u \|_{Z_h} < C_* h^{ \min\{ p + \vartheta, p+1-\frac{\omega}{2} \}} $ is used. Choosing $\bar{h}$ and $\varrho_U$ small enough, we obtain due to Lemma \ref{lem:linear_wellposed} well-posedness of the variational problem
\[
b_h'(\hat{u}_h, z_h, u_{h,*}) = \hat{\ell}_h(z_h)  \quad \forall z_h \in Z_h  
\]
with an $h$-independent a priori estimate $\| \hat{u}_h \|_{Z_h} \leq c_{*}^{-1} \| \hat{\ell}_h \|_{Z_h'}$ for all $h < \bar{h}$.
If $h<\bar{h} < 1$ and if we deduce $\kappa = \frac{1}{c_*}$ from the apriori estimate,
it follows that $G'_{h, u_{h,*}}\colon Z_h \to Z_h'$ is invertible for all $u_* \in D_{h, *}$.

To establish the Hölder estimate of $G_h'$, we proceed analogously to estimate \eqref{eq:coersive_n} and combine it again with an inverse inequality to obtain for $u_1,u_2 \in D_{h,\star}$ and all $u_h,z_h \in Z_h$
\[
\begin{split}
  | b'_h(u_h,z_h;u_1) -  \ell_h(z_h) - &b'_h(u_h,z_h;u_2) + \ell_h(z_h) |
  \\
  &= | n'_h(u_h,z_h;u_1 - u_2)  | \\
&\leq C_n c_\iptheta^{-2} \max_{[0,T]} \disnorm{ u_1 - u_2 }_{h} \| z_h \|_{Z_h} \| u_h \|_{Z_h} \\ 
&\leq C_n c_\iptheta^{-3} C_\textrm{inv} \sqrt{c_r} 2 h^{- \frac{ \omega}{2}} \| u _1 - u_2 \|_{Z_h} \| z_h \|_{Z_h} \| u_h \|_{Z_h} \\
&\leq C_n c_\iptheta^{-3} C_\textrm{inv} \sqrt{c_r} 2 h^{- \frac{ \omega}{2}} \| u _1 - u_2 \|_{Z_h}^{1- \chi} \| u _1 - u_2 \|_{Z_h}^\chi \| z_h \|_{Z_h} \| u_h \|_{Z_h} \\
&\leq  C_n c_\iptheta^{-3} C_\textrm{inv} \sqrt{c_r} h^{- \frac{\omega}{2}} (2 C_* h^{\min\{ p + \vartheta, p+1-\frac{\omega}{2} \} } )^{1 - \chi} \| u _1 - u_2 \|_{Z_h}^{\chi}  \| z_h \|_{Z_h} \| u_h \|_{Z_h} \\
&\leq \lambda  \| u _1 - u_2 \|_{Z_h}^{\chi}  \| z_h \|_{Z_h} \| u_h \|_{Z_h},
\end{split}
\]
so that the $\chi$-Hölder continuity holds with $\lambda \leq C_n c_\iptheta^{-3} C_\textrm{inv} \sqrt{c_r}  (2 C_*)^{1 - \chi}$ independent of $h$, due to the assumption $ (1- \chi) \min\{ p + \vartheta, p+1-\frac{\omega}{2} \} \geq \frac{\omega}{2}$ and $\bar{h} < 1$.

In order to fulfill assumption \eqref{betaK}, we establish an estimate on $  \| G(u_{h,*}) \|_{Z_h'}$.
%\leq C_G \| u_{h,*} \|_{Z_h}$$ with $C_G>0$ independent of $h$. 
% can be made sufficiently small.
%For $b_h(u_{h,*}, \cdot)$ we have the following estimates,
Let $u_* \in V$ denote a function such that $\pi_h u_* = u_{h,*} \in D_{\hbar, *}$.

Then by triangle inequality and as in the proof of Lemma \ref{le:convergence_1}
	\[
\begin{split}
  	\vert m_h (u_{h,*}, z_h ) \vert 
	 &=  \bigg\vert \sum_{j=1}^k \int_{t_{j-1}}^{t_j} (  \M \dt u_{h,*} , z_h )_{\Omega_h} \wrtdt 
	+ (M \jump{ u_{h,*}}_{j-1} , z_{h,j}(t_{j-1}) )_{\Omega_h} \bigg\vert
	\\
	&= \bigg\vert \sum_{j=1}^k - \int_{t_{j-1}}^{t_j} (  \M u_{h,*} , \dt z_h )_{\Omega_h} \wrtdt 
	- (M (u_{h,*})_{j}(t_j) , \jump{z_h}_{j}  )_{\Omega_h} \bigg\vert
	\\
	&\leq \bigg\vert \sum_{j=1}^k - \int_{t_{j-1}}^{t_j} (  \M  ( u_{h,*} - u_*) , \dt z_h )_{\Omega_h} \wrtdt 
	- (M (u_{h,*})_{j}(t_j) , \jump{z_h}_{j}  )_{\Omega_h} \bigg\vert + | (M \dt u_*, z_h)_Q |
	\\
	&= \left\vert \sum_{j=1}^k (M^\half (u_{h,*} )_{j}(t_j) , M^\half \jump{z_h}_{j}  )_{\Omega_h}\right\vert  + | (M \dt u_*, z_h)_Q | \\
	& \leq C_M \left(\| u_{h,*} \|_{H^1(0,T;L^2(\Omega))} + \| \dt u_* \|_{L^2(0,T;L^2(\Omega))} \right) \| z_h \|_{Z_h},
\end{split}
\]
where we use a trace like inequality on $(u_{h,*})_j$, so that time evaluations can be estimated by the $H^1$ norm. As in the proof of Lemma \ref{le:convergence_1} we have
\[
\begin{split}
	\vert  \bilinspacesf(u_{h,*}, \numpvtest)  \vert 
	&\leq \int_0^T   \|  \jacobian u_{h,*}  \|_{L^2(\Omega_h)}  \|\numpvtest \|_{L^2(\Omega_h)}
	+ 2 C_A  \| u_{h,*} \|_{L^2(\partial \Omega_h)} 
	\| A_\norvec \jump{\numpvtest} \|_{L^2(\partial \Omega_h)}  \wrtdt 
	\\ 
	&\leq C_{\tild{A}} ( \| u_{h,*} \|_{Z_h} + \| u_{h,*} \|_{L^2(0,T;L^2(\partial \Omega_h))} ) \| z_h \|_{Z_h},
\end{split}
\]
and
\[
\begin{split}
	| \bilinspaceip( u_{h,*}, \numpvtest) |  &\leq \int_0^T 
	( \| \D^\half \jacobian u_{h,*} \|_{L^2(\Omega_h)}   ) 
	\| D^\half \jacobian \numpvtest \|_{L^2(\numspacedomain)} 
	+ | \bilinspaceipbdr( u_{h,*}, \numpvtest) | \wrtdt  
	\\
	&\leq \int_0^T (1 + 2 C_\textrm{ip} D_{\max}^2)  \disnorm{u_{h,*}}_h  \disnorm{z_h}_h \wrtdt
	\\
	&\leq C_D  \| u_{h,*} \|_{Z_h} \| z_h \|_{Z_h},
\end{split}
\]
as well as from estimate \eqref{eq:coersKant}
\[
\begin{split}
| n_h(u_{h,*}, z_h)  | = \half | n'_h(u_{h,*}, z_h; u_{h_*}) | &\leq \half | n'_h(u_{h,\star}, z_h; u_{h,*} - u) | + \half | n'_h(u_{h,*}, z_h; u) | \\
&\leq C_N \| u_{h,*} \|_{Z_h} \| z_h \|_{Z_h}
\end{split}
\]
with $C_N \leq C_n ( C_\textrm{inv} \sqrt{c_r} c_\iptheta^{-3} C_* + C_U \varrho_U )$. Combining the estimates, we conclude with the discrete Poincaré inequality \eqref{eq:poincare1} and triangle inequality 
\[ \label{eq:phi_estimate}
\begin{split}
 \| G(u_{h,*}) \|_{Z_h'} &\leq C_G ( \| u_{h,*} \|_{Z_h} + \| u_{h,*} \|_{L^2(0,T;L^2(\partial \Omega_h))}  \\
 &+ \| \dt u_{h,*} \|_{L^2(0,T;L^2(\Omega))} + \| \dt (u_{h,*} - u_* ) \|_{L^2(0,T;L^2(\Omega))} )
\end{split}
\]
with $C_G>0$ independent of $h$.
Note that by triangle inequality \[ \label{eq:estimate_Z_h}
\| u_{h,*} \|_{Z_h} \leq \|  u_{h,*} - u \|_{Z_h} + \| u \|_{Z_h} \leq C_* \bar{h}^\xi + C_U \varrho_u,
\]
and in combination with a discrete trace inequality
\[ \label{eq:bounaryL2}
\begin{split}
\| u_{h,*} \|_{L^2(0,T;L^2(\partial \Omega_h))} 
&\leq C_\textrm{tr} h^{- \half }\| u_{h,*} - u \|_{L^2(0,T;L^2( \Omega_h))} + \| u \|_{L^2(0,T;L^2(\partial \Omega_h))} \\
&\leq C_\textrm{tr} h^{- \half } c_p c_{\iptheta}^{-1} \| u_{h,*} - u \|_{Z_h} + C_U \varrho_U %+ \| u \|_{L^2(0,T;L^2(\partial \Omega_h))} 
\\
&\leq C_\textrm{tr} c_p c_{\iptheta}^{-1} C_* \bar{h}^{ \min\{ p + \vartheta, p+1-\frac{\omega}{2} \} - \half} + C_U \varrho_U .
\end{split}
\]
Using an inverse inequality for the time derivative we have
\[ \label{eq:timederi}
\begin{split}
	\|\dt u_* \|_{L^2(0,T;L^2(\Omega))} &\leq C_\textrm{inv} \tau^{-1} \|  u_{h,*} - u \|_{L^2(0,T;L^2(\Omega))}  + \| \dt u \|_{L^2(0,T;L^2)}
	% &\leq C_\textrm{inv} \tau^{-1} \| u_*- u_{h,*} \|_{L^2(0,T;L^2(\Omega))}  + \tau^{-1} C_* \bar{h}^{\xi - \omega} + C_U \varrho_U
	\\
	&\leq C_\textrm{inv} \sqrt{c_r} c_p c_{\iptheta}^{-1} C_* \bar{h}^{\min\{ p + \vartheta - \omega , p + 1- \frac{3 \omega}{2} \}} + C_U \varrho_U.
\end{split}
\] 
Using an inverse inequality for the time derivative and projection error estimate \eqref{eq:error-projL2} we have
\[ \label{eq:timederi2}
\begin{split}
	\|\dt (u_{h,*} - u_* ) \|_{L^2(0,T;L^2(\Omega))} &\leq  C_\textrm{inv} \tau^{-1 } \|  u_* - u_{h,*} \|_{L^2(0,T;L^2(\Omega))} 
	% &\leq C_\textrm{inv} \tau^{-1} \| u_*- u_{h,*} \|_{L^2(0,T;L^2(\Omega))}  + \tau^{-1} C_* \bar{h}^{\xi - \omega} + C_U \varrho_U
	\\
	&\leq C_\textrm{inv} \sqrt{c_r} C_{\tild{\pi}} T^{-\half} \| u_*\|_V \bar{h}^{\min\{p+1,\nu \} - \omega}  
\end{split}
\] 
with $C_{\tild{\pi}}>0$ independent of $h$.
\begin{comment}
	content...
\[ \label{eq:timederi}
\begin{split}
\|\dt u_* \|_{L^2(0,T;L^2(\Omega))} &\leq  C_\textrm{inv} \tau^{-1 } ( \|  u_* - u_{h,*} \|_{L^2(0,T;L^2(\Omega))} + \|  u_{h,*} - u \|_{L^2(0,T;L^2(\Omega))} ) + \| \dt u \|_{L^2(0,T;L^2)}
% &\leq C_\textrm{inv} \tau^{-1} \| u_*- u_{h,*} \|_{L^2(0,T;L^2(\Omega))}  + \tau^{-1} C_* \bar{h}^{\xi - \omega} + C_U \varrho_U
\\
&\leq C_\textrm{inv} \sqrt{c_r} (  C_{\tild{\pi}} T^{-\half} \| u_*\|_V \bar{h}^{\min\{s,p+1,q+1\} - \omega} + c_p c_{\iptheta}^{-1} C_* \bar{h}^{\xi - \omega} )+ C_U \varrho_U 
\end{split}
\] 
\end{comment}
%Here, we rely on the assumption that $\xi > \omega $
%p,q \geq 1$ if $\omega>1$ 
%in order to guarantee that $\xi - \omega>0$ and ${\min\{s,p+1,q+1\} - \omega}>0$.
%choose $r \in (0 , \frac{1}{2 \kappa^2 \lambda} )$, so that
Therefore, choosing $\bar{h}$, $\varrho_U$ and $\varrho_\ell$ small enough, we get by combining estimates \eqref{eq:phi_estimate}, \eqref{eq:estimate_Z_h}, \eqref{eq:bounaryL2}, \eqref{eq:timederi}, \eqref{eq:timederi2}
\[
 \kappa \lambda \| - \Gamma_* ( b_h(u_{h,\star}, \cdot) -  \ell_h(\cdot) )  \|_{Z_h}^\chi \leq \kappa^{1+\chi}\lambda \| b_h(u_{h,\star}, \cdot) -  \ell_h(\cdot) \|_{Z_h'}^\chi  < \chi_0   
\]
independent of $h$.
%As the final step, we verify whether $r^{-} \leq \varrho$. Due to the choice of $\varrho$,
%we have $c_\varrho = \frac{c_{\iptheta}}{2}$ and $\varrho = \frac{c_{\iptheta}}{2 \lambda}$. Therefore, the condition transforms into
%\[
 %1 - \sqrt{1 - 2 \kappa \lambda \| \Gamma_* G(u_*) \|_X}  < \kappa \lambda \frac{c_{\iptheta}}{2 \lambda} = 1,
%\] 
%which holds, due to inequality \eqref{betaK}.

Therefore, there exists a $\bar{h}>0$ such that the required inequalities hold for any $u_{h,*} \in D_{\bar{h},*}$.
We are now left to choose a starting value $u_{h,*}$ such that $B^{Z_h}_{\varrho}(u_{h,*})\subseteq D_{*,\bar{h}}$ with $\varrho$ from equation \eqref{rstart_rstarstar}. For this we fix $\bar{h}$ and let $u_{h,*} \in D_{h,*} $ with $h$ to be chosen. Assuming that $\upsilon=\kappa \lambda \phi^\chi$ is below a certain threshold cf. \eqref{betaK}, we have the bound \[
\varrho \leq 2 \phi (1 - \kappa \lambda \phi^\chi) \leq 2 \phi = 2 \| - \Gamma(b_h(u_{h,*}, \cdot) - \ell_h(\cdot))\|_{Z_h'}
\]
 so that $\varrho$ can be made sufficiently small if ${h}$, $\varrho_U$ and $\varrho_\ell$ are small enough by the previous estimates \eqref{eq:phi_estimate} to \eqref{eq:timederi}. 
Let $z_h \in B_\varrho^{Z_h}(u_{h, *})$. Then 
\[
\| z_h - u \|_{Z_h} \leq \| z_h - u_{h,*} \|_{Z_h} + \| u_{h,*} - u \|_{Z_h} \leq \varrho + C_* h^\xi ,
\]
\begin{comment}
	content..
\[
\| z_h - u \|_{L^2(0,T;L^2(\partial \Omega_h))} \leq C_\textrm{tr} h^{-\half} \| z_h - {u_{h,*}} \|_{L^2(0,T;L^2(\Omega_h))} + \| u_{h,*} - u \|_{L^2(0,T;L^2(\partial \Omega_h))} \leq C_\textrm{tr} h^{-\half} \varrho + C_* h^{\xi - \omega},
\]
\[
 \| \dt (z_h - u) \|_{L^2(0,T;L^2(\Omega))} \leq C_\textrm{inv} \tau^{-1}  \| z_h -  u_{h,*} \|_{L^2(0,T;L^2(\Omega))}  +  \| \dt ( u_{h,*} - u) \|_{L^2(0,T;L^2(\Omega))} \leq C_\textrm{inv} h^{-\omega} \varrho + C_* h^{\xi - \omega}.
\]
\end{comment}
so that choosing $\tild{h}<\bar{h}$, $\varrho_U$ and $\varrho_\ell$ so small such that $ \varrho + C_* h^\xi < C_* \bar{h}^\xi$, we obtain $B^{Z_h}_{\varrho}(u_{h,*})\subseteq D_{*,\bar{h}}$ for all $h<\tild{h}$.

Applying Lemma \ref{NewtonKantorowich} gives a unique $u_h \in B_{\tild{\varrho}}^{Z_h}(u_{h,*}) \cap D_{h,*}$ with $\tild{\varrho}$ from \eqref{rstart_rstarstar} such that $G_h(u_h)=0$ for all $h<\tild{h}$. %, where $r^+,\varrho$ are not depending on $h$. 
Since $u_h \in D_{h,*}$, we automatically obtain the error estimate by construction.
%Finally, we remark that $\ell_h$ from definitions \eqref{eq:l_h}, \eqref{eq:l_h_boundary} can be made sufficiently small, assuming that the data \eqref{eq:data} are small enough and a suitable interpolation is used. 
\end{proof}
\begin{comment}
	content...

\begin{remark}
Since $\G_h$ is a quadratic map, convergence results for Newton's methods for quadratic maps hold. For other convergence results and error estimates on Newton's methods, we refer to \cite{Deuflhard2011} and the references therein. 
\end{remark}
\end{comment}
\begin{remark}
We note that the assumptions $\| u \|_U \leq \varrho_U$ and $\| \ell_h \|_{Z_h'} \leq \varrho_\ell$ can be fulfilled, if the data of the continuous problem \eqref{eq:data} are chosen small enough, cf. \cite{QuadraticWave}. Since uniform space-time methods are equivalent to time stepping methods, local existence of a discrete solution always holds if $\tau$ is chosen small enough. However, in Theorem \ref{th:nonlinear_wellposed} we show global in time well-posedness, in which case a convergence analysis without smallness assumptions \eqref{eq:smallness_ass} is not possible, since then global existence of the continuous solution $u$ is not clear. 
%Since space-time methods can be split into small time intervals, choosing $\tau$ and the final time $T$ small enough the convergence analysis automatically holds. 
\end{remark}
\begin{remark}
In comparison to Theorem \ref{th:convergence} we can no longer take $p=0$, since then the assumptions are not fulfilled. Taking $p=q= \omega = \vartheta = 1$ with $s=2$, $\chi=\half$ all assumptions are fulfilled and we obtain convergence if $$u \in H^{2}(0,T;H^1(\Omega;\R^{1+d})) \cap W^{2, \infty}(0,T;L^2(\Omega;\R^{1+d}) ) \cap C([0,T];H^2(\Omega;\R^{1+d}).$$ 
%Taking linear finite elements, we have $\xi=1$ and therefore obtain linear convergence in $h$, if the solution $u$ is smooth enough and $h=\tau$, so that $\omega=1$.
%The conditions $\| \ell_h \| \leq \varrho_\ell$ and $p,q \leq 1$ if $\omega>1$ can be omitted if one assumes $ \xi (1 - \chi) > \frac{\omega}{2} $. Then $\lambda$ can be made sufficiently small to establish estimate \eqref{betaK}. %One has to be careful that the condition 
A convergence analysis for less regular solutions is not contained in this paper.
\end{remark}

\clearpage

\section{Numerical experiments}\label{se:experiments}
For the numerical experiments, we aim to model sound waves propagating through soft tissue. To this end, we choose
%as nonlinearity parameter $\frac{B}{A} = 5$. This leads to the following choice of parameters 
%We solve a system without units. Neumann boundary condition for pressure and Dirichlet for velocity. The system gets existed by a Gauß-Pulse in pressure, i.e
%$ p_{source} = b 10^{-5} e^{- a t^2(-x^2 - y^2)}$ 
%but more stable by initial conditions actually.
%Typical is $ p_0 = 5 \times 10^6$ Pa.
%$$ \epsilon = \frac{p_0}{\rho_0 c_0^2} = \sqrt{ \frac{2 I}{\rho_0 c_0^3}} = 2.2 \times 10^{-3} $$
%\subsection{Diagnostic ultrasound}
%We calculate with a frequency of $30 \ Mhz$ and initial power of $0.1 \ Wcm^{-2}$. In this case the mach number is $\epsilon = 2.7 \times 10^{-5}$. %So $p_0 = 7,6 \times 10^4$. Then
%Physical calculations: $\omega_{r} = 5 MHz$, $c_0 = 1500$, $\rho_0 = 1000$, $\mu_S = 0.001$ for water or $\mu_S = 0.05$ for oil.
%This gives $Rey^{-1} = \frac{\omega_r \mu_S}{c_0^2 \rho_0} = 2.2 \times 10^{-6}$ for water or $Rey^{-1} = 1.1 \times 10^{-4}$ .
%Need $\nu := \frac{4}{3} + \frac{\mu_B}{\mu_S} = 2.3$ due to lacking values. $\zeta = \nu Rey^{-1} = 2.6 \times 10^{-4}$. $\lambda=\frac{\gamma_r}{Pr}$, we have $Pr = 7$ for water and $Pr=100$ for oil. $\gamma_{heat} = \frac{c_p}{c_v} = \frac{4184}{4157} = 1.006$ for water. Now $\gamma_r = \frac{\gamma_{heat}}{\gamma_{heat}-1} = 1.7 \times 10^2.$ This would give $\lambda = 24$ for water or $1.7$ for oil.
%Based on these estimates we choose
\begin{align}
	\alpha=1, \ 
	\beta= 10^{-4}, \ 
	\gamma = 6, \
	\delta = 1, \
	\varepsilon = 1, \
	\zeta = 10^{-4}, \
	\eta = 1 , \
	\theta = 1.
\end{align}
Since system \eqref{eq:pvmodel} is in nondimensional form, most parameters are set to one. The value $\gamma=6$ is derived from the nonlinearity parameter $\frac{B}{A} \approx 5$ in tissue, see \cite{Hamilton:1997} for details. The damping parameters are chosen to approximately match the viscosity coefficients observed in tissue. 
Due to computational constraints, we restrict our study to the case $d=2$. The approximate domain $\Omega_h \times I_h$ is discretized using regular Lagrange quadrilateral elements with $4 \tau = h$. The mesh is refined such that $h=2^{-n_r}$, $n_r \geq 1$. We set the penalty constant to $C_\sigma = 50$. In our experiments, the linear systems to be solved exhibit a block-diagonal structure, so we solve them using GMRES in combination with a block Gauss–Seidel preconditioner.
The implementation is carried out in M++, see \cite{baumgarten2021}, and can run in parallel. The complete source code is available at
 \hyperlink{https://gitlab.kit.edu/kit/mpp/mpp}{\tt https://gitlab.kit.edu/kit/mpp/mpp}. 
\subsection{Error \corr{indicator} for $p$-adaptivity}	
In this paper, we only consider residual based adaptivity. We remark that a duality based adaptivity with respect to a goal functional is also possible, but not investigated in this paper. We apply $p$-adaptivity as a variant of \newcorr{ the maximum marking strategy } described in \cite[Algorithm 1]{DoerflerFindeisenWienersZiegler:2019} and \cite[Chapter 4.2]{ziegler:diss}, with the additional feature that
derefinement is included. \newcorr{Thus, the degree of a space-time cell is increased if $\tilde{\varsigma}_{R} \ge \vartheta_{\mathrm{ref}} \max_{R' \in \mathcal{R}_h}\tilde{\varsigma}_{R'}$
and decreased if $\tilde{\varsigma}_{R} \le \vartheta_{\mathrm{deref}} \max_{R' \in \mathcal{R}_h}\tilde{\varsigma}_{R'}$, depending on $\vartheta_{\mathrm{ref}} > \vartheta_{\mathrm{deref}} > 0$. } 
In \cite{DoerflerFindeisenWienersZiegler:2019}, for the problem $ (M \dt +  A) u = f$ the following local error estimator for a space-time cell $R=(I_j \times K)$ is used
\[
\begin{split}
\varsigma_{R}^2 = h \| ( M \dt + A ) u_h - f \|^2_{L^2(I \times K)}
+ \half \sum_{j=0}^k \| \M^{\frac{1}{2}} \jump{u_h}_{j} \|_{L^2(K)}^2 + \half \int_I \| A_\norvec \jump{u_h} \|_{L^2(\partial K)}^2 \wrtdt .
\end{split}
\]
We, therefore, use the following error indicator to account for the nonlinear terms
\[
\begin{split}
	\tilde{\varsigma}_{R}^2 =\  &h \| ( M \dt + A ) u_h +  N(u_h, u_h) - f \|^2_{L^2(I \times K)}
	\\
	&+ \half \sum_{j=0}^k \| \M^{\frac{1}{2}} \jump{u_h}_{j} \|_{L^2(K)}^2 + \half \int_I \| A_\norvec \jump{u_h} \|_{L^2(\partial K)}^2 \wrtdt 
\end{split}
\]
and neglect diffusion terms, since they are small in our application. When advancing to the next refinement level, the coarse solution is used as the initial guess for Newton's method. For each refined cell, both the temporal and spatial polynomial degrees are increased by one.

\clearpage

\subsection{Known solution}
For the first experiment, we choose the solution 
\[ \label{eq:knownsol} 
u(t,x,y) =  \bigl(
	\dt \psi , \,
	\partial_x \psi, \,
	\partial_y \psi
	\bigr)
\text{ ~ with ~}\psi(t,x,y) = \psi_A \sin( \varphi t) \sin( k x) \sin( k y) 
\] and set $\Omega=(0,1)^2$, $I=(0,1)$, $\neu=\emptyset$, $\psi_A=0.01$, $\varphi=6
\pi$, $k=\pi$.
%[TO DO: Convergence Plot in the $H^2$ Norm, but also this $\| \cdot \|_{Z_h}$ norm maybe?]
In this setting, we have $\vartheta=\omega=1$. Hence, Theorem \ref{th:nonlinear_wellposed} gives for $q=p=1$ a convergence rate of $\xi = 1$, for $q=1,p=2$ a rate $\xi = \frac{3}{2}$ and for $p=q=2$ we get the convergence rate $\xi =2$. 

When plotting the errors \newcorr{with respect to the DG-norm defined
  in \eqref{eq:DGNorm}} of the discrete solutions against the known
solution $u$ in \eqref{eq:knownsol}, we observe that the actual
convergence rates generally exceed the predictions of Theorem
\ref{th:nonlinear_wellposed}, see Figure \ref{fig:dgerror}. As
expected, the $L^2$ errors are consistently smaller than the DG-norm
errors.

\begin{figure}[h]
	\centering
	\includegraphics[scale=0.8]{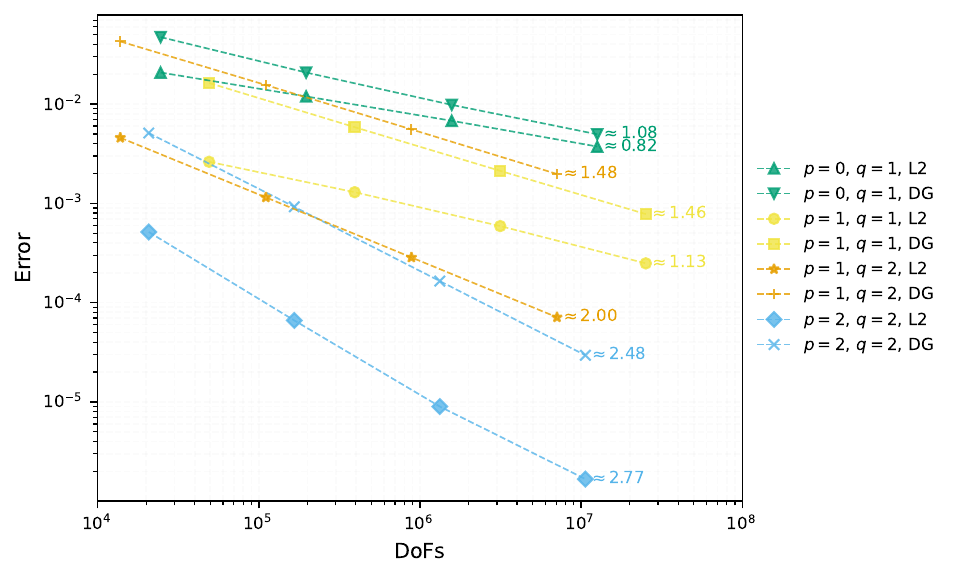}
	\caption{DG-norm error and $L^2$ space-time error comparison between known and discrete solutions for different time degrees $p$ and space degrees $q$ in dependence to the total degrees of freedom (DoFs) used. The extrapolated convergence order (ECO) is indicated at the right end of each line.}
	\label{fig:dgerror}
\end{figure}

\clearpage

\subsection{Unknown solution}
For the next experiment, we set $\Omega=\newcorr{(}-\half, \half\corr{)}^2$, $I=(0,1)$, $\Gamma_N=\emptyset$ and assume homogeneous boundary conditions. As source terms, we choose a Gaussian bump excitation oscillating harmonically in time
\[
p_S(t,x,y) =
\begin{cases}
	a \exp(1) \, \sin\!\big( 2 \pi a_f t \big)
	\, \exp\!\!\left(
	-\dfrac{1}{1 - \dfrac{x^2 + y^2}{r_{\mathrm{a}}^2}}
	\right)
	,
	& \text{if } x^2 + y^2 < r_{\mathrm{a}}^2, \\[1.2em]
	0, & \text{otherwise,}
\end{cases}
\]
for the pressure 
and set $v_S(t,x,y) = 0$. The initial conditions $p_0,v_0$ are set to zero and as parameters we choose $a=2$, $a_f=3$, and $r_a=0.2$. \newcorr{Fig. \ref{fig:sol} \corr{and \ref{fig:sol2} show} the pressure variable of the numerical solution at different time steps. As expected, the solution is radially symmetric.}  
\begin{figure}[h]
	\centering
	\includegraphics[height=12em]{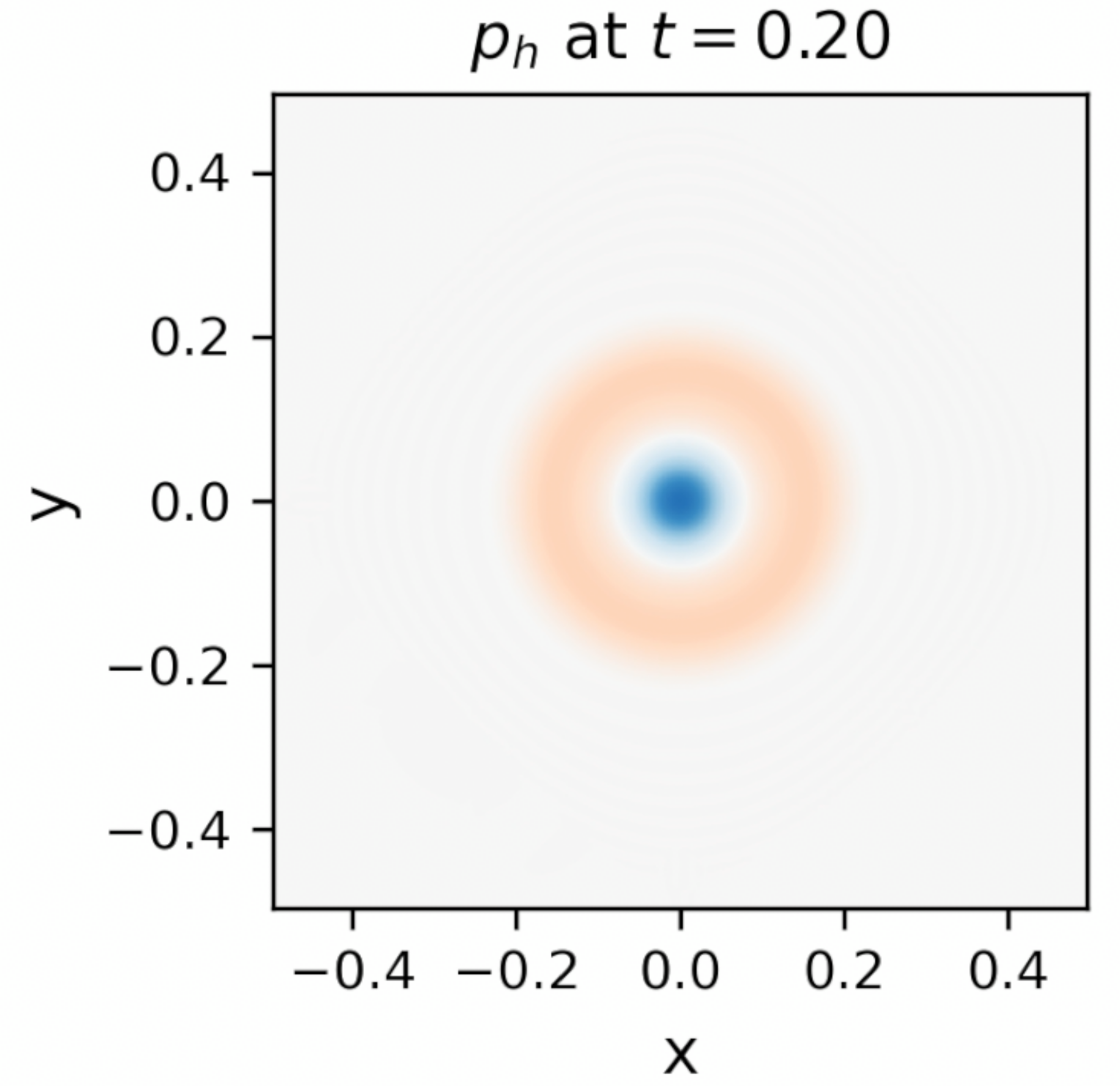}
	\includegraphics[height=12em]{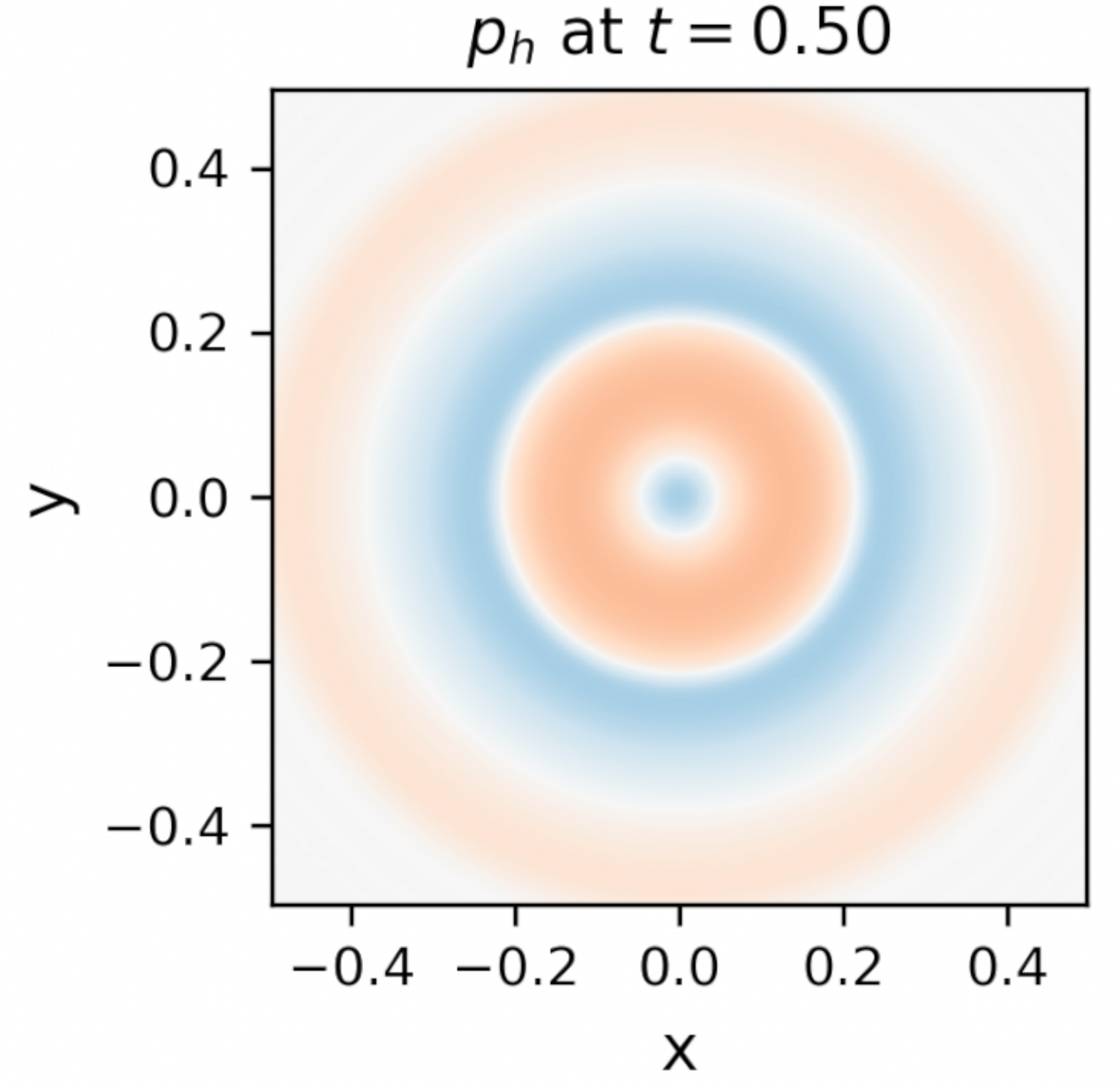}
	\includegraphics[height=12em]{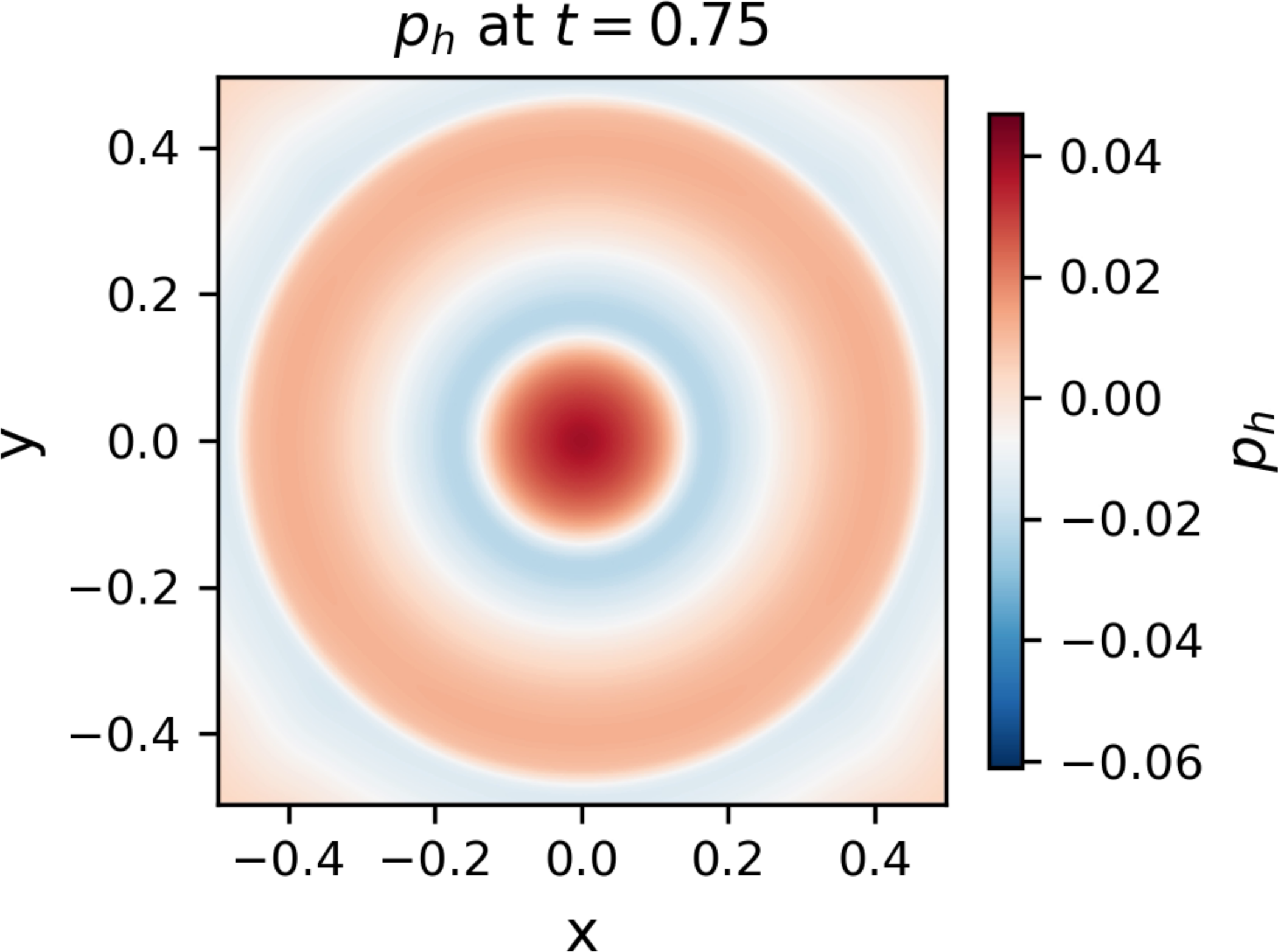}
	\caption{Numerical solution for $n_r = 7$ and $p = q = 1$ at different time steps. The approximated pressure $p_h$ is plotted.} \label{fig:sol}
\end{figure}
\begin{figure}[h] 
	\centering
	\includegraphics[height=16em]{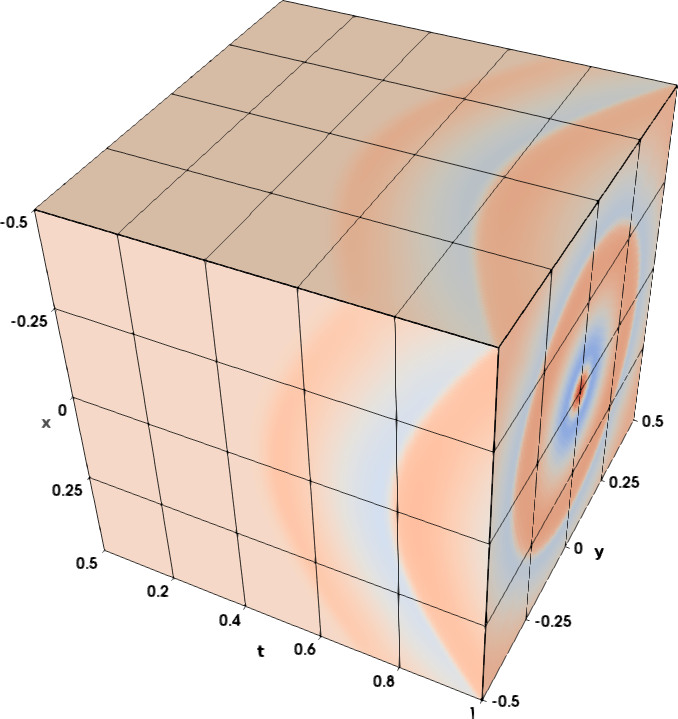}
        \qquad
        \qquad
        \qquad
	\includegraphics[height=16em]{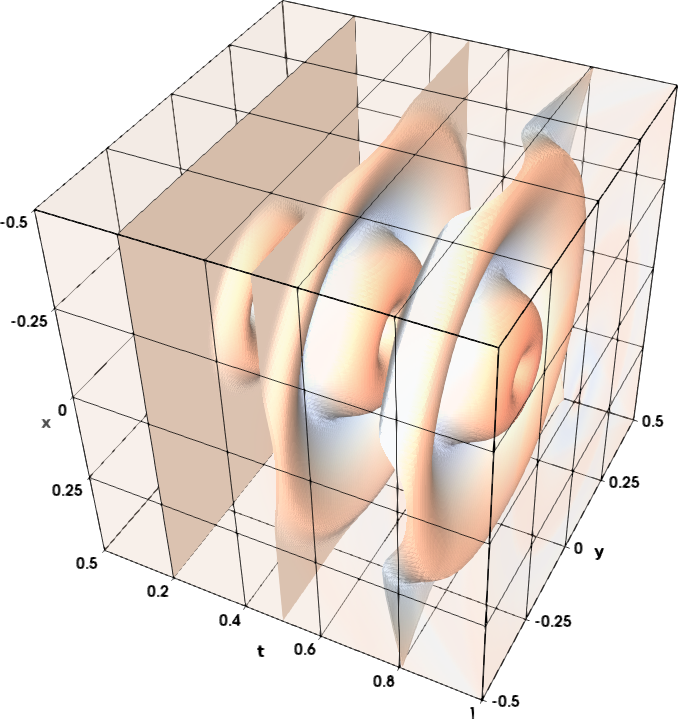}
	\caption{\corr{Numerical solution $p_h$ for $n_r = 7$ and $p = q = 1$ on the space-time domain.}}
        \label{fig:sol2}
\end{figure}

We first note that the model successfully reproduces the characteristic waveform observed in nonlinear acoustics, including the generation of higher harmonics at integer multiples of the base frequency $a_f = 3$, see Figure~\ref{fig:compare}. 
To compare the linear and nonlinear behavior of the system, we compute, for $n_r = 7$ and $p = q = 1$, the numerical solution $u_h$ of the full nonlinear model and the corresponding solution $u_h^L$ obtained by setting $\gamma = \delta = \eta = \theta = 0$. 
To analyze the harmonic content, we evaluate the discrete Fourier transform $\mathcal{F}_h$ of the pressure signals. 

Since the source $p_s$ has a base frequency $a_f=3$, additional spectral components appear at  frequencies $ \omega_{\mathcal{F}} = 3, 6, 9, \dots$, indicating the presence of nonlinear harmonic generation in the full model. We also observe the well known behavior that the nonlinear wave exhibits a shift in its peak values, corresponding to a local change in sound speed. Namely, peaks occur earlier than in the linear wave when the amplitude is positive, and later when the amplitude is negative.
\begin{figure}[h]
	\centering
	\includegraphics[scale=0.8]{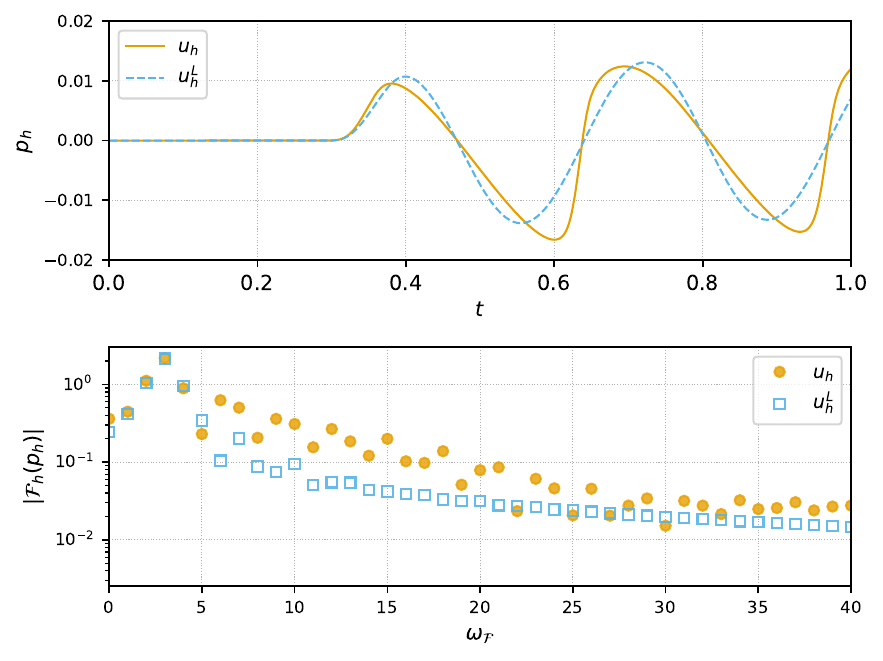}
	\caption{At the top, the time evolution of the pressure at $(x,y) = (0.246094,\, 0.246094)$ is shown. 
		The blue dotted line corresponds to the pressure component of the solution 
		$u_h^L = (p_h^L, v_h^L)$ obtained with $\gamma = \delta = \eta = \theta = 0$, 
		while the orange solid line represents the pressure from the full nonlinear model. 
		At the bottom, the corresponding magnitudes of the discrete Fourier transforms 
		of $p_h$ and $p_h^L$ are displayed.}
	\label{fig:compare}
\end{figure}

To test the adaptivity and the error estimator described above, we plot the error estimator $\tild{\varsigma_R}$ in the $L^2$ norm for each iteration and for different values of $n_r$ in dependence of the total degrees of freedoms used in Figure \ref{fig:errorestimator}. We perform two refinement steps, starting with initial space-time degrees of $p=0,q=1$ and $p=q=1$.
\newcorr{The marking parameter $\vartheta_{\mathrm{ref}}$ is set to $10^{-1}$ and the derefinement parameter $\vartheta_{\mathrm{deref}} = 10^{-2}$. }
We observe that the error estimator decays within each refinement in all cases, except for $p=1,q=1,n_r=6.$
\begin{figure}[h]
	\centering
	\includegraphics[scale=0.7]{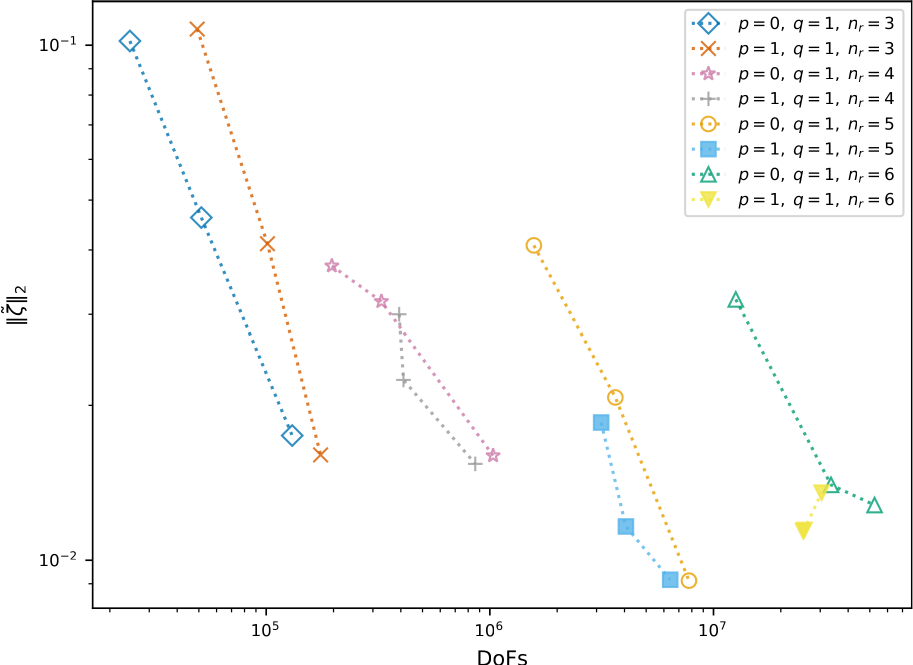}
	\caption{The $L^2$ norm of the nonlinear global error estimator $\tild{\varsigma}$ is shown in dependence of the degree of freedoms of each adaptive iteration.}
	\label{fig:errorestimator}
\end{figure}
%In Figure \ref{fig:massconv} one can see visually that the functions converge, when refining. From the tables we observe that $P$ converges fast then $E$ and that for higher polynomial degree $E$ tends to converge slower.

Next, in order to test convergence, we evaluate the functionals
\[
P(p)[t]:= \int_\Omega p(t) \, \mathrm{d}x \text{ and } E(v)[t]:= \int_\Omega v(t) \cdot v(t) \,  \mathrm{d}x
\]
numerically at each time point and compute the $L^1$ and $L^2$ errors in time. Since the exact solution is unknown, we use a discrete reference solution $u_h^R=(p_h^R, v_h^R)$ computed with $p=q=1$ and $n_r=7$. The results for the functionals $P$ and $E$ are shown in the Tables \ref{tab:1}, \ref{tab:2} as well as in the Figures \ref{fig:masscon} and \ref{fig:ekincon}, respectively. Overall, we observe convergence of the functionals toward the values evaluated at the discrete reference solution. However, we note that adaptivity only becomes advantageous for relatively high numbers of degrees of freedom, around $10^7$. 
We also observe the tendency that adaptivity approximates $E$ more effectively than $P$.

\begin{table}[h]
	\centering
	\renewcommand{\arraystretch}{1.2}
	\begin{tabular}{|c|ccc|ccc|}
		\hline
		cells
		& \multicolumn{3}{c|}{$p=q=1$}
		& \multicolumn{3}{c|}{$p=q=2$} \\
		\cline{2-7}
		& DoF & $\lVert P(e_h) \rVert_{1}$ & $\lVert E(e_h) \rVert_{2} $
		& DoF & $\lVert P(e_h) \rVert_{1}$ & $\lVert E(e_h) \rVert_{2} $\\
		\hline
		2\,048   	& 49\,150 & 1.67613 & 0.17467 & 165\,900 & 0.47892 & 0.02186 \\
		16\,380     & 393\,200 & 0.36469 & 0.04925 & 1\,327\,000 & 0.16742 & 0.00780 \\
		131\,100    & 3\,146\,000 & 0.21867 & 0.02271 & 10\,620\,000 & 0.10059 & 0.01722 \\
		1\,049\,000 & 25\,170\,000 & 0.06514 & 0.01115 & 84\,930\,000 & 0.08683 & 0.00636 \\
		\hline
		EOC & & 1.562 & 1.323 & & 0.821 & 0.594 \\
		\hline
	\end{tabular}
	\caption{Approximation errors of $P,E$ and extrapolated order of convergence (EOC) for $p=q=1$ and $p=q=2$. } \label{tab:1}
	\label{tab:eoc_P_E}
\end{table}

\begin{table}[h]
	\centering
	\renewcommand{\arraystretch}{1.2}
	\begin{tabular}{|c|ccc|ccc|}
		\hline
		cells
		& \multicolumn{3}{c|}{adaptive $p=0, \, q=1$}
		& \multicolumn{3}{c|}{adaptive $p=q=1$} \\
		\cline{2-7}
		& DoF & $\lVert P(e_h) \rVert_{1}$ & $\lVert E(e_h) \rVert_{2} $
		& DoF & $\lVert P(e_h) \rVert_{1}$ & $\lVert E(e_h) \rVert_{2} $\\
		\hline
		2\,048      & 130\,700 & 0.47985 & 0.08535 & 175\,100 & 0.51269 & 0.08465 \\
		16\,380     & 1\,034\,000 & 0.21338 & 0.03840 & 860\,500 & 0.22487 & 0.03840 \\
		131\,100    & 7\,769\,000 & 0.08215 & 0.00803 & 6\,403\,000 & 0.07407 & 0.00808 \\
		1\,049\,000 & 52\,490\,000 & 0.05237 & 0.00312 & 30\,430\,000 & 0.02869 & 0.00331 \\
		\hline
		EOC & & 1.1016 & 1.656 & & 1.656 & 1.846 \\
		\hline
	\end{tabular}
	\caption{Approximation errors of $P,E$ and extrapolated order of convergence (EOC) for adaptive runs starting with $p=0, q=1$ and $p=q=1$.}
\label{tab:2}
\end{table}

\clearpage

\corr{In Tab.~\ref{tab:eoc_P_E} and Tab.~\ref{tab:2} the extrapolated order of convergence 
is investigated for the weighted errors $\lVert P(e_h) \rVert_{1} = 10^4
  \int_0^T | P(p_h) - P(p_h^R) | \wrtdt$ and $\lVert E(e_h) \rVert_{2} =
  10^4 \left(\int_0^T ( E(v_h) - E(v_h^R))^2 \wrtdt \right)^\half$.}

\begin{figure}[h]
	\centering
	\includegraphics[scale=0.7]{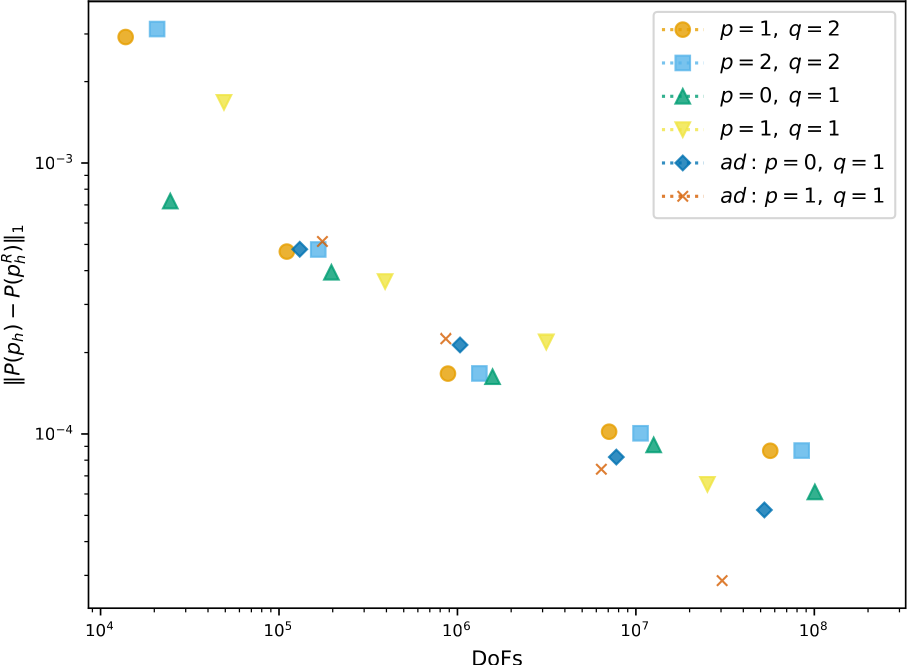}
        \\[-3mm]
	\caption{Convergence to the functional $P$ in the $L^1$ norm is shown in dependence of the degrees of freedom of each solution. For adaptive solutions the last iteration is used. }
	\label{fig:masscon}
\end{figure}

\begin{figure}[h]
	\centering
	\includegraphics[scale=0.7]{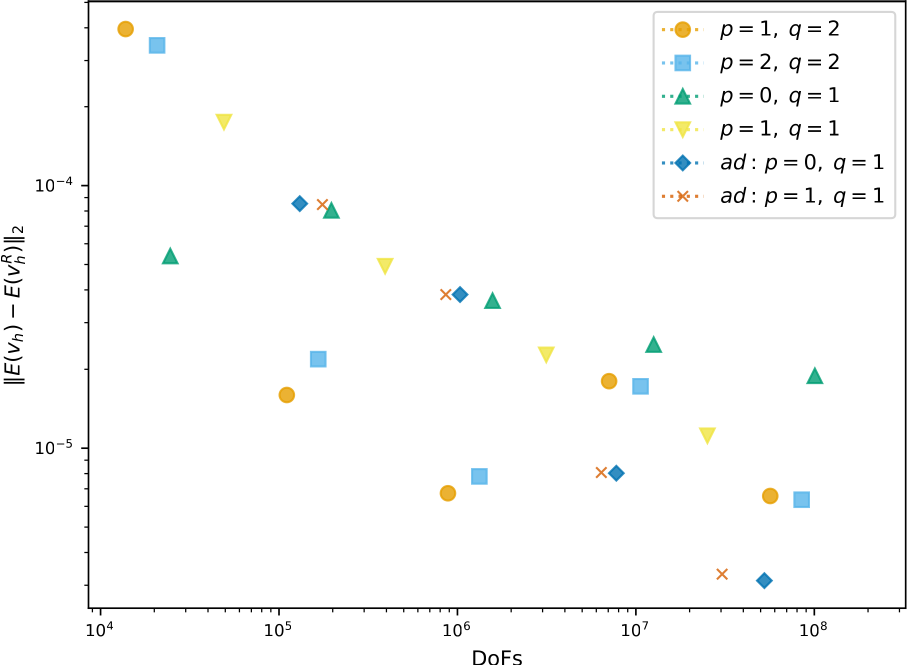}
        \\[-3mm]
	\caption{Convergence to the functional $E$ in the $L^2$ norm is shown in dependence of the degrees of freedom of each solution. For adaptive solutions the last iteration is used. }
	\label{fig:ekincon}
\end{figure}

\clearpage

\section{Conclusion} \label{se:conclusion}
In this work, we have developed a space-time DG method for a nonlinear acoustic equation, extending the framework of linear symmetric Friedrichs systems. We have rigorously established the well-posedness of the scheme and proved convergence in the natural DG norm. Numerical experiments confirm the theoretical results and illustrate the potential of 
$p$-adaptivity for improving accuracy. While adaptivity shows benefits for high resolution cases, its effectiveness at moderate resolutions remains limited, highlighting the need for further investigation and optimization. Overall, the proposed approach provides a robust and flexible framework for the numerical simulation of nonlinear acoustic waves, paving the way for future studies.

A mathematically rigorous convergence analysis of the 
$p$-adaptive algorithm is beyond the scope of this paper and remains an open question in the literature. Future investigations could explore $hp$-adaptive algorithms or hybrid DG methods to further enhance efficiency.
\section*{Acknowledgment}
The authors thank Barbara Kaltenbacher for her careful reading of the manuscript and helpful suggestions, which have led to marked improvements. 
This research was funded in whole or in part by the Austrian Science Fund (FWF) [10.55776/P36318]. For open access purposes, the author has applied a CC BY public copyright license to any author accepted
manuscript version arising from this submission.
\nocite{}
\bibliographystyle{plain}
\bibliography{lit}
\appendix
\section{Details to upwind flux} \label{se:upwindflux}
%\subsubsection*{Derivation of $\Aup$}
In order to obtain an expression for the upwind flux \eqref{eq:def_upwindflux} $\Aup$ of
\begin{align}
	\para \dt p + \divv \vvar &= p_{S} \\
	\pare \dt \vvar + \grad p &= \vvar_{S},
\end{align}
we follow along the lines of the theory presented in \cite[Section 3]{DoerflerFindeisenWienersZiegler:2019}
and assume $d=2$ for simplicity. In its general form, the Friedrichs system can be written as 
\begin{align}
	(\M \dt + A) u = f 
\end{align}
defined on $\Omega \subset \R^2$, where
\begin{align}
	M =
	\begin{pmatrix}
		\para & 0 & 0 \\
		0 & \pare & 0 \\
		0 & 0 & \pare
	\end{pmatrix}
\end{align}
and $A = A_1 \dx{1} + A_2 \dx{2} $
with
\begin{align}
	A_1 = \begin{pmatrix}
		0 & 1 & 0 \\
		1 & 0 & 0 \\
		0 & 0 & 0
	\end{pmatrix},
	\quad
	A_2 = \begin{pmatrix}
		0 & 0 & 1 \\
		0 & 0 & 0 \\
		1 & 0 & 0
	\end{pmatrix}, \quad
	A_\norvec = n_1 A_1 + n_2 A_2 = \begin{pmatrix}
		0 & n_1 & n_2 \\
		n_1 & 0 & 0 \\
		n_2 & 0 & 0
	\end{pmatrix} .
\end{align}
To construct $\Aup$, we use the nontrivial eigenpairs of $M^{-1} A_\norvec$ fulfilling $A_\norvec w = \lambda M w$ with $\lambda \in \R$ and $w \in \R^3$. 
Writing $c_0:=\sqrt{\alpha^{-1} \epsilon^{-1}}$ these are
$$ 
\lambda_{2,3} = \pm c_0, \quad w_{2,3} = \begin{pmatrix}
	\alpha^{-1} \\ \pm c_0 \norvec 
\end{pmatrix}.
$$
According to \cite{DoerflerFindeisenWienersZiegler:2019}, $\Aup$ is then given as
$$\Aup y =  A_{\norvec_K}^{-} y  $$
with
\begin{align} A_\norvec^{-} y = \lambda_3 \frac{w_3 \cdot M y}{w_3 \cdot M w_3} M w_3 
	&= \frac{-c}{\alpha^{-1} + \varepsilon c_0^2} \begin{pmatrix}
		1 \\ - \varepsilon c_0 \norvec
	\end{pmatrix} (y_1 - c_0 \varepsilon \norvec \cdot (y_2,y_3)) \\
	&= \frac{-1}{\alpha^{-1} c_0^{-1} + \varepsilon c_0} \begin{pmatrix}
		1 \\ - \varepsilon c_0 \norvec
	\end{pmatrix} (y_1 - c_0 \varepsilon \norvec \cdot (y_2,y_3)) \\
	&= \frac{-1}{2 \sqrt{\alpha^{-1} \varepsilon}} \begin{pmatrix}
		1 \\ - \epsilon c_0 \norvec
	\end{pmatrix} (y_1 - c_0 \varepsilon \norvec \cdot (y_2,y_3)).
\end{align}
Therefore, writing $Z_0=\sqrt{\alpha^{-1} \varepsilon}$ for the acoustic impedance, we get
$$ \Aup \begin{pmatrix}
	\jump{p}_{F,K} \\ \jump{v}_{F,K}
\end{pmatrix} = 
- \frac{1}{2 Z_0} ( \jump{p}_{F,K} - Z_0 \jump{v}_{F,K} \cdot \norvec_K) \begin{pmatrix}
	1 \\ - Z_0 \norvec_K
\end{pmatrix} .$$
This yields the matrix
\[
\Aup = \half \begin{pmatrix}
	- \frac{1}{Z_0} & \norvec_K^\transpose \\
	\norvec_K & - Z_0 \norvec_K \norvec_K^\transpose 
\end{pmatrix}
= \AD + \half A_{\norvec_K} \text{ with } \AD = \half \begin{pmatrix}
	- \frac{1}{Z_0} & 0 \\
	0 & - Z_0 \norvec_K \norvec_K^\transpose 
\end{pmatrix} ,
\]
as in \eqref{eq:def_upwindflux}.
%\subsubsection*{Proof of coersivity of $a_h$}
\begin{lemma}
	We have for all $u_h \in Z_h$ 
	\[
	a_h(u_h,u_h) = \half \sum_{\element \in \allelements} \left(  \sum_{\face \in \faceset_\element} \| \sqrt{-\AD} \jump{u_h}_{F,K} \|_F^2 \right). \]
\end{lemma}
\begin{proof}
	%Proof that this choice satisfies \eqref{eq:upwindflux_assump} ...
	By partial integration we have for all $\element \in \allelements$
	\[ ( \A \numpv, \numpvtest)_{\element} + (\numpv, \A \numpvtest)_\element = (A_{\norvec_\element} \numpv, \numpvtest )_{\partial \element},  \]
	which corresponds to the skew-adjointness of $\A$. Setting $z_h = u_h$ yields
	\[ \label{eq:A_skew}
	( \A \numpv, \numpv)_{\element} = \half (A_{\norvec_\element} \numpv, \numpv )_{\partial \element} = \half \sum_{F \in \faceset_K} (A_{\norvec_\element} u_{h,K}, u_{h,K} )_{F} = \sum_{F \in \faceset_K} ( p_{h,K}, v_{h,K} \cdot \norvec_K )_{F}.
	\]
	For inner faces we have
	\[ \label{eq:A_inner}
	\begin{split}
	  \half \sum_{\element \in \allelements} &
          \sum_{\face \in \faceset_\element \cap \Omega} \| \sqrt{-\AD} \jump{u_h}_{F,K} \|_{L^2(F)}^2  
		\\
		&=  \sum_{\element \in \allelements} \sum_{\face \in \faceset_\element \cap \Omega} \half (  -\AD \jump{u_h}_{F,K}, \jump{u_h}_{F,K} )_F 
		\\
		&=  \sum_{\element \in \allelements} \sum_{\face \in \faceset_\element \cap \Omega} ( \half \AD \jump{u_h}_{F,K},  u_{h,K} - u_{h,K_F}  )_F  
		\\
		&=  \sum_{\element \in \allelements}\sum_{\face \in \faceset_\element \cap \Omega} ( \half \AD \jump{u_h}_{F,K},  u_{h,K} - u_{h,K_F}  )_F +  (\half \AD \jump{u_h}, u_{h,K})_F - (\half \AD \jump{u_h}, u_{h,K})_F 
		\\
		&=  \sum_{\element \in \allelements} \sum_{\face \in \faceset_\element \cap \Omega} ( \AD \jump{u_h}_{F,K},  u_{h,K}  )_F - ( \half \AD \jump{u_h}_{F,K}, u_{h,K_F}  )_F - (\half \AD \jump{u_h}, u_{h,K})_F 
		\\
		&=  \sum_{\element \in \allelements} \sum_{\face \in \faceset_\element \cap \Omega} ( \AD \jump{u_h}_{F,K},  u_{h,K}  )_F
		\\
		&= \sum_{\element \in \allelements}  \sum_{\face \in \faceset_\element \cap \Omega}( \AD \jump{u_h}_{F,K},  u_{h,K}  )_F + \half (A_{\norvec_K} \jump{u_h}_{F,K}, u_{h,K} ) -  \half (A_{\norvec_K} \jump{u_h}_{F,K}, u_{h,K} )	
		\\
		&= \sum_{\element \in \allelements}  \sum_{\face \in \faceset_\element \cap \Omega} ( \Aup \jump{u_h}_{F,K},  u_{h,K}  )_F -  \half (A_{\norvec_K} \jump{u_h}_{F,K}, u_{h,K} )	
		\\
		&= \sum_{\element \in \allelements}  \sum_{\face \in \faceset_\element \cap \Omega} ( \Aup \jump{u_h}_{F,K},  u_{h,K}  )_F +  \half (A_{\norvec_K} u_{h,K}, u_{h,K} )_F -  \half ( A_{\norvec_K} u_{h,K_F}, u_{h,K})_F 
		\\
		&= \sum_{\element \in \allelements} \sum_{\face \in \faceset_\element \cap \Omega} ( \Aup \jump{u_h}_{F,K},  u_{h,K}  )_F +  \half (A_{\norvec_K} u_{h,K}, u_{h,K} )_F.
	\end{split}
	\]
	From definition \eqref{eq:def_jumpboundary}, we have in the Dirichet case
	\[ \label{eq:A_dir}
	\begin{split}
	  \sum_{\element \in \allelements} &
          \sum_{\face \in \faceset_\element \cap \dir} \| \sqrt{-\AD} \jump{u_h}_{F,K} \|_{L^2(F)}^2  
		\\
		&=  \sum_{\element \in \allelements} \sum_{\face \in \faceset_\element \cap \dir} (  -\AD \jump{u_h}_{F,K}, \jump{u_h}_{F,K} )_F %+ \frac{1}{8} ( \A_{\norvec_K} \jump{u_h}_{F,K}, \jump{u_h}_{F,K} )_F
		\\
		&=  \sum_{\element \in \allelements} \sum_{\face \in \faceset_\element \cap \dir} ( - \AD u_{h,F}, u_{h,F} )_F 
		\\
		&=  \sum_{\element \in \allelements} \sum_{\face \in \faceset_\element \cap \dir} ( - \AD u_{h,F}, u_{h,F} )_F - \half ( \A_{\norvec_K} u_{h,F}, u_{h,F} )_F + \half ( \A_{\norvec_K} u_{h,F}, u_{h,F} )_F
		\\
		&= \sum_{\element \in \allelements}  \sum_{\face \in \faceset_\element \cap \dir} ( \Aup \jump{u_h}_{F,K},  u_{h,K}  )_F + \half ( \A_{\norvec_K} u_{h,F}, u_{h,F} )_F
	\end{split}
	\]
	and in the Neumann case
	\[ \label{eq:A_neu}
	\begin{split} 
	  \half \sum_{\element \in \allelements}
          &
          \sum_{\face \in \faceset_\element \cap \neu} \| \sqrt{-\AD} \jump{u_h}_{F,K} \|_{L^2(F)}^2  
		\\
		&=  \sum_{\element \in \allelements} \sum_{\face \in \faceset_\element \cap \neu} \half (  -\AD \jump{u_h}_{F,K}, \jump{u_h}_{F,K} )_F %+ \frac{1}{8} ( \A_{\norvec_K} \jump{u_h}_{F,K}, \jump{u_h}_{F,K} )_F
		\\
		&=  \sum_{\element \in \allelements} \sum_{\face \in \faceset_\element \cap \neu} ( - 2 \AD (0,v_{h,F} )^\transpose, u_{h,F} )_F
		\\
		&=  \sum_{\element \in \allelements} \sum_{\face \in \faceset_\element \cap \neu} ( - 2 \AD (0,v_{h,F} )^\transpose, u_{h,F} )_F - ( \A_{\norvec_K} (0,v_{h,F} )^\transpose, u_{h,F} )_F + \half ( \A_{\norvec_K} u_{h,F}, u_{h,F} )_F
		\\
		&= \sum_{\element \in \allelements}  \sum_{\face \in \faceset_\element \cap \neu} ( \Aup \jump{u_h}_{F,K},  u_{h,K}  )_F + \half ( \A_{\norvec_K} u_{h,F}, u_{h,F} )_F .
	\end{split}
	\]
	In total, combing \eqref{eq:A_inner}, \eqref{eq:A_dir}, \eqref{eq:A_neu}, and using skew-adjointness of $\A$ \eqref{eq:A_skew}, we obtain
	\[ \label{eq:A_coersiv}
	\begin{split}
		&\half \sum_{\element \in \allelements} \left(  \sum_{\face \in \faceset_\element} \| \sqrt{-\AD} \jump{u_h}_{F,K} \|_F^2 \right) \\
		& =  \sum_{\element \in \allelements} \left(  \sum_{\face \in \faceset_\element} ( \Aup \jump{u_h}_{F,K},  u_{h,K}  )_F + \half ( \A_{\norvec_K} u_{h,F}, u_{h,F} )_F \right)
		\\
		&= \sum_{\element \in \allelements} \left(  (\A u_h, u_h)_K + \sum_{\face \in \faceset_\element} ( \Aup \jump{u_h}_{F,K},  u_{h,K}  )_F \right)
		=a_h(u_h,u_h) .
	\end{split}
	\]
\end{proof}
To obtain an estimate as \eqref{eq:upwindflux_assump}, we note that
\[
\sqrt{-\AD} = \frac{1}{\sqrt{2}}
\begin{pmatrix}
	\frac{1}{\sqrt{Z_0}} & 0 \\
	0 & \sqrt{Z_0} \norvec_K \norvec_K^\transpose 
\end{pmatrix},
\]
since $\sqrt{\norvec_K \norvec_K^\transpose} = \norvec_K \norvec_K^\transpose$. Thus,
\[
\| - \sqrt{\AD} u \|^2_{L^2(F)} = \frac{1}{2 Z_0} \| p \|^2_{L^2(\face)} + \frac{Z_0 }{2}\| v \cdot n  \|_{L^2(\face)}^2 
\]
together with
\[
\| \A_{\norvec_K} u \|^2_{L^2(F)} = \| p \|^2_{L^2(\face)} + \| v \cdot n  \|_{L^2(\face)}^2 
\]
 gives us
\[
2 \min \{\frac{1}{ Z_0}, Z_0\} \| \A_{\norvec_K} u \|^2_{L^2(F)} \leq \|  \sqrt{- \AD} u \|^2_{L^2(F)} \leq \half \max\{\frac{1}{ Z_0}, Z_0\} \| \A_{\norvec_K} u \|^2_{L^2(F)} .
\]
Therefore, due to equation \eqref{eq:A_coersiv}, we see that estimate \eqref{eq:upwindflux_assump} holds with
$c_A = 4 \min \{\frac{1}{ Z_0}, Z_0\}$  and $C_A = \frac{1}{4} \max\{\frac{1}{ Z_0}, Z_0\} $. 
%\subsubsection{Symmetry}
\begin{lemma} \label{le:upwind_symetric}
	For all $u_h \in Z_h + U$ and $z_h \in Z_h$ we have
	\[
	\sum_{\element \in \allelements} \sum_{\face \in \faceset_\element} (\Aup \jump{u_h}_{F,K}, z_{h,K})_F = \sum_{\element \in \allelements} \sum_{\face \in \faceset_\element} (u_{h,K}, \Aup \jump{z_h}_{F,K})_F.
	\]
\end{lemma}
\begin{proof}
	First of all, we note that $-\Aup$ has eigenpairs
	\[ 0,
	\begin{pmatrix}
		Z_0 n_1 \\
		1 \\
		0
	\end{pmatrix} \quad
	0, 
	\begin{pmatrix}
		Z_0 n_2 \\
		0 \\
		1
	\end{pmatrix} \quad
	\half \left( Z_0 + \frac{1}{Z_0} \right),
	\begin{pmatrix}
		-\frac{1}{Z_0} \\
		n_1 \\
		n_2
	\end{pmatrix} .
	\] 
	Therefore, it is possible to choose a square root of $-\Aup$.
	For inner faces, due to the symmetry of $\Aup$, we have
	\[
	\begin{split}
	  \sum_{\element \in \allelements} &\sum_{\face \in \faceset_\element \cap \Omega} (\Aup \jump{u_h}_{F,K}, z_{h,K})_F \\
		&= 	\sum_{\element \in \allelements} \sum_{\face \in \faceset_\element \cap \Omega} ( \jump{u_h}_{F,K}, \Aup z_{h,K})_F + (u_{h,K}, \Aup {z_{h,K_F}})_F - (u_{h,K}, \Aup {z_{h,K_F}})_F
		\\
		&= 	\sum_{\element \in \allelements} \sum_{\face \in \faceset_\element \cap \Omega} ( u_{h,K},   \Aup \jump{z_h}_{F,K})_F + (u_{h,K_F}, \Aup {z_{h,K}})_F - (u_{h,K}, \Aup {z_{h,K_F}})_F
		\\
		&= 	\sum_{\element \in \allelements} \sum_{\face \in \faceset_\element \cap \Omega} ( u_{h,K},   \Aup \jump{z_h}_{F,K})_F - (u_{h,K_F}, - \Aup {z_{h,K}})_F + (u_{h,K}, - \Aup {z_{h,K_F}})_F \\
		&= 	\sum_{\element \in \allelements} \sum_{\face \in \faceset_\element \cap \Omega} ( u_{h,K},   \Aup \jump{z_h}_{F,K})_F - ( \sqrt{ - \Aup } u_{h,K_F},  \sqrt{ - \Aup } {z_{h,K}})_F 
		\\
                &\qquad\qquad\qquad\qquad\qquad\qquad\qquad\quad\, 
                + ( \sqrt{ - \Aup } u_{h,K}, \sqrt{ - \Aup } {z_{h,K_F}})_F \\
		&= 	\sum_{\element \in \allelements} \sum_{\face \in \faceset_\element \cap \Omega} ( u_{h,K},   \Aup \jump{z_h}_{F,K})_F .
	\end{split}
	\]
	For boundary faces, one can use the symmetry of $\Aup$.
\end{proof}
\end{document}